\documentclass[11pt]{articlefederico2017}
\usepackage{graphicx}
\usepackage{amsfonts }
\usepackage{amsmath}
\usepackage{fullpage}
\usepackage{amssymb}
\usepackage{amsthm}
\usepackage{tikz}
\usepackage{lipsum}
\usepackage{mathrsfs}
\usepackage{hyperref}
\usepackage{float}
\usepackage{multicol,cleveref}
\usepackage{enumerate}
\usepackage{blkarray}
\usepackage{mathtools}

\usepackage{xcolor}

 \definecolor{darkslategray}{rgb}{0.18, 0.31, 0.31}

\definecolor{ganador}{HTML}{D8345F}

\usepackage{subcaption}

\theoremstyle{theorem}
\newtheorem{theorem}{Theorem}[section]
\newtheorem{corollary}[theorem]{Corollary}

\newtheorem{proposition}[theorem]{Proposition}

\theoremstyle{definition}
\newtheorem{definition}[theorem]{Definition}

\newtheorem{example}[theorem]{Example}

\makeatletter
\newtheorem*{rep@theorem}{\rep@title}
\newcommand{\newreptheorem}[2]{%
\newenvironment{rep#1}[1]{%
 \def\rep@title{#2 \ref{##1}}%
 \begin{rep@theorem}}%
 {\end{rep@theorem}}}
\makeatother
\newreptheorem{theorem}{Theorem}

\newcommand{\Z}{{\mathbb{Z}}}
\newcommand{\R}{{\mathbb{R}}}

\newcommand{\E}{\overline{E}}

\newcommand{\1}{\overline{1}}
\newcommand{\2}{\overline{2}}
\newcommand{\3}{\overline{3}}
\newcommand{\4}{\overline{4}}
\newcommand{\5}{\overline{5}}
\newcommand{\6}{\overline{6}}
\newcommand{\7}{\overline{7}}
\newcommand{\n}{\overline{n}}
\newcommand{\ii}{\overline{i}}
\newcommand{\BB}{\mathsf{B}}
\newcommand{\RR}{\mathsf{R}}
\renewcommand{\r}{\mathsf{r}}

\newcommand{\bt}{\bar{t}}

\newcommand{\conv}{\mathsf{conv}}
\newcommand{\cone}{\mathsf{cone}}
\renewcommand{\int}{\mathsf{int}}
\newcommand{\des}{\mathsf{des}}
\newcommand{\asc}{\mathsf{asc}}
\newcommand{\rev}{\mathsf{rev}}

\renewcommand{\deg}{\mathsf{deg}}

\renewcommand{\d}{{\mathsf{d}}}

\newcommand{\e}{{\mathsf{e}}}
\newcommand{\f}{{\mathsf{f}}}
\newcommand{\g}{{\mathsf{g}}}
\newcommand{\uu}{{\mathsf{u}}}
\newcommand{\vv}{{\mathsf{v}}}
\newcommand{\ww}{{\mathsf{w}}}
\newcommand{\one}{{\mathsf{1}}}

\newcommand{\B}{{\mathsf{B}}}
\newcommand{\M}{{\mathsf{M}}}
\newcommand{\N}{{\mathsf{N}}}
\newcommand{\T}{{\mathsf{T}}}

\newcommand{\NN}{{\mathcal{N}}}

\definecolor{darkgreen}{rgb}{0,0.5,0}
\definecolor{brown}{rgb}{0.5,0.3,0}
\definecolor{darkblue}{rgb}{0,0,0.8}
\definecolor{darkred}{rgb}{0.5,0,0}

\definecolor{lightmagenta}{rgb}{.5,0,.5}

\definecolor{magenta}{rgb}{1,0,1}
\definecolor{cyan}{rgb}{0,0.9,0.9}
\definecolor{grey}{rgb}{0.5,0.5,0.5}
\definecolor{lightgrey}{rgb}{0.7,0.7,0.7}
\definecolor{yellow}{rgb}{0.9, 0.9, 0}

\newcommand{\magenta}{\textcolor{magenta}}
\newcommand{\cyan}{\textcolor{cyan}}
\newcommand{\grey}{\textcolor{grey}}

\begin{document}

\title{
The bipermutahedron}

\author{
\textsf{Federico Ardila\footnote{\textsf{San Francisco State University, Universidad de Los Andes; federico@sfsu.edu. Partially supported by NSF grant DMS-1855610 and Simons Fellowship 613384.}}}
}

\date{}

\maketitle

%
%
\begin{abstract} 
The harmonic polytope and the bipermutahedron are two related polytopes which arose in Ardila, Denham, and Huh's work on the Lagrangian geometry of matroids. We study the bipermutahedron. We show that its faces are in bijection with the vertex-labeled and edge-labeled multigraphs with no isolated vertices; the generating function for its $f$-vector is a simple evaluation of the three variable Rogers--Ramanujan function. 

We show that the $h$-polynomial of the bipermutahedral fan is the biEulerian polynomial, which counts bipermutations according to their number of descents.
We construct a unimodular triangulation of the product $\Delta \times \cdots \times \Delta$ of triangles that is combinatorially equivalent to (the triple cone over) the bipermutahedral fan. Ehrhart theory then gives us a formula for the biEulerian polynomial, which we use to show that this polynomial is real-rooted and that the $h$-vector of the bipermutahedral fan is log-concave and unimodal.

We describe all the deformations of the bipermutahedron; that is, the ample cone of the bipermutahedral toric variety. We prove that among all polytopes in this family, the bipermutahedron has the largest possible symmetry group. Finally, we show that the Minkowski quotient of the bipermutahedron and the harmonic polytope equals 2.

\end{abstract}

\section{Introduction}

Motivated by the Lagrangian geometry of conormal varieties,
in joint work with Graham Denham and June Huh \cite{ArdilaDenhamHuh1}, we introduced the \emph{conormal fan} $\Sigma_{\M, \M^\perp}$ of a matroid $\M$ -- a Lagrangian counterpart of the Bergman fan $\Sigma_\M$ \cite{AK}. 
We used the conormal fan $\Sigma_{\M, \M^\perp}$ to give new geometric interpretations of the Chern-Schwartz-MacPherson cycle of a matroid $\M$ \cite{LRS} and of the $h$-vectors of the broken circuit complex $BC(\M)$ and independence complex $I(\M)$ of $\M$. Combined with tools from combinatorial Hodge theory, we used this geometric framework to prove that these $h$-vectors are log-concave,  as conjectured by Brylawski and Dawson \cite{Brylawski, Dawson} in the early 1980s. 

In our construction of the conormal fan $\Sigma_{\M,\M^\perp}$ with Denham and Huh, we encountered two related polytopes associated to a positive integer $n$: the \emph{harmonic polytope} $H_{n,n}$ and the \emph{bipermutahedron} $\Pi_{n,n}$. In particular, 
the conormal fans $\Sigma_{\M,\M^\perp}$ of all matroids $\M$ on $[n]$ live inside a fan called the \emph{bipermutahedral fan} $\Sigma_{n,n}$, and the fact that this fan is projective -- that is, the existence of the bipermutahedron --  is a fundamental step in our proof of Brylawski and Dawson's log-concavity conjectures \cite{ArdilaDenhamHuh1}.

The harmonic polytope $H_{n,n}$ is studied in our paper \cite{ArdilaEscobar} with Laura Escobar. The bipermutahedron $\Pi_{n,n}$ is the main object of study of this paper. Its name derives from the fact that its vertices are in bijection with the \emph{bipermutations} of $[n]$, which are the sequences of length $2n-1$ containing one element of $[n]$ exactly once and every other element of $[n]$ exactly twice.

Our main results are the following:

\begin{itemize}
\item Proposition \ref{prop:fvector} shows that the $(d-2)$-faces of the $n$th bipermutahedron $\Pi_{n,n}$ are in bijection with the multigraphs on vertex set $[d]$ and edge set $[n]$ with no isolated vertices.

\item Theorem \ref{th:fseries} shows that the generating function for the face numbers of  bipermutahedra is a simple evaluation of the three variable Rogers--Ramanujan function.

\item Theorem \ref{th:biEulerian} shows that the $h$-polynomial of the bipermutahedral fan $\Sigma_{n,n}$ is the $n$th biEulerian polynomial, which enumerates bipermutations according to their number of descents.

\item Theorem \ref{th:triangulation} constructs a unimodular triangulation of the product $\Delta^n$ of $n$ standard triangles that is combinatorially isomorphic to (a triple cone over) the bipermutahedral fan $\Sigma_{n,n}$.

\item
Theorem \ref{th:biEulerianseries} uses the Ehrhart theory of $\Delta^n$ to express the $n$th biEulerian polynomial $B_n(x)$ as the numerator of the generating function of the sequence ${k \choose 2}^n$.

\item
Theorem \ref{th:realrooted} shows that the biEulerian polynomial $B_n(x)$ is real-rooted, and hence that  the $h$-vector of the bipermutahedral fan is log-concave and unimodal.

\item
Proposition \ref{prop:symmetry} shows that among the polytopes whose normal fan is the bipermutahedral fan $\Sigma_{n,n}$, the bipermutahedron $\Pi_{n,n}$ has the largest possible symmetry group.

\item
Proposition \ref{prop:deformations} describes all the polytopes whose normal fan is the bipermutahedral fan $\Sigma_{n,n}$. This is the ample cone of the bipermutahedral toric variety $X_{\Sigma_{n,n}}$.

\item
Theorem \ref{th:Pi/H} shows that the Minkowski quotient of the bipermutahedron and the harmonic polytope is $\Pi_{n,n}/H_{n,n} = 2$ in any dimension. 
This is the largest $\lambda$ for which $\lambda H_{n,n}$ is a Minkowski summand of $\Pi_{n,n}$.

\end{itemize}

\section{The bipermutahedral fan and the bipermutahedron} \label{sec;bipermutahedron}

In this section we recall the definition of the bipermutahedron and its (inner) normal fan, as introduced in  \cite{ArdilaDenhamHuh1}. Throughout the paper we fix a positive integer $n \geq 2$, and write $E=\{1,\ldots, n\}$.

\begin{definition}\label{def:biperm}
A \emph{bipermutation} on $E$ is a sequence $\B=b_1|\ldots|b_{2n-1}$ of elements of $E$, such that
\begin{enumerate} \itemsep 3pt
\item one element $k(\B)=k$ of $E$ appears exactly once in $\B$, and
\item every other element $i \neq k$ of $E$ appears exactly twice in $\B$, 
\end{enumerate}
\end{definition}

We will sometimes write $\B$ by writing the special element $k$ in bold, and writing $\overline{i}$ for the second occurrence of $i$ for each $i \neq k$. For example, we rewrite the bipermutation $2|3|2|1|3$ as $2|3|\2|{\bf 1}|\3$. We will use these two notations interchangeably.

There is a bijection between the permutations on $[n]$ and the permutations of $\{1,1,2,2,\ldots, n,n\}$: given a bipermutation $\B$ on $[n]$ whose non-repeated element is $k(\B) = k$, simply add another  $k$ at the end of $\B$. Therefore there are $(2n)!/2^n$ bipermutations on $[n]$.

We consider two copies of $\R^n$ with standard bases  $\{\e_i \, : \, i \in [n]\}$ and $\{\f_i \, : \, i \in [n]\}$, respectively. We also consider their dual spaces, with dual bases $\{\e^i \, : \, i \in [n]\}$ and $\{\f^i \, : \, i \in [n]\}$
For any subset $S$ of $[n]$, we write 
\[
\e_S=\sum_{i \in S} \e_i, \qquad \f_S=\sum_{i \in S} \f_i,
\]
and similarly for $\e^S$ and $\f^S$.
We also consider the pair of dual $(n-1)$-dimensional vector space 
\[
\M_n = \{ x \in \R^n \, : \, \sum_i x_i = 0\}, \qquad 
\N_n := \mathbb{R}^n/\R\e_E
\]
The bipermutahedron and its normal fan live in $\M_n \times \M_n$ and in $\N_n \times \N_n$, respectively. We begin by introducing the normal fan, which plays a central role in the Lagrangian geometry of matroids, since it contains the conormal fan of every matroid on $[n]$ \cite{ArdilaDenhamHuh1}.

\subsection{The bipermutahedral fan}

Let $p=(p_1,\ldots,p_n)$ be an $E$-tuple of points in $\mathbb{R}^2$. The \emph{supporting line} of $p$, denoted $\ell(p)$,  is the lowest  line of slope $-1$ containing a point in $p$. For each point $p_i$, the vertical and horizontal projections of $p_i$ onto $\ell(p)$ will be labelled $i$.
The \emph{bisequence of $p$}, denoted $\B(p)$,  is obtained by reading the labels on $\ell(p)$ from right to left. See Figure \ref{fig:config} for an illustration.

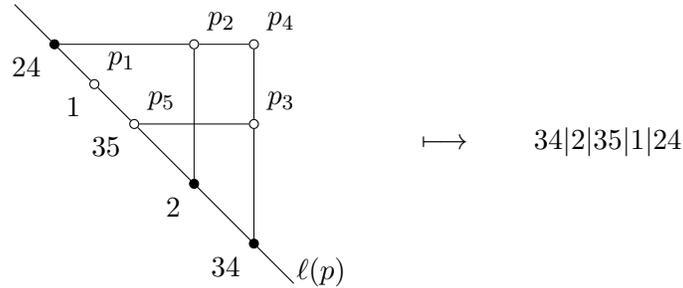
\begin{figure}[h]
\[
\begin{tikzpicture}[scale=0.75,baseline=(current bounding box.center),
plain/.style={circle,draw,inner sep=1.2pt,fill=white},
root/.style={circle,draw,inner sep=1.2pt,fill=black}]
\begin{scope}[rotate=-45]
\node (1) at (1,0) [root,label=below left:$24$] {};
\node (2) at (2,0) [plain,label=below left:$1$,label=above right:$p_1$] {};
\node (3) at (3,0) [plain,label=below left:$35$,label=above right:$p_5$] {};
\node (4) at (4.5,0) [root,label=below left:$2$] {};
\node (5) at (6,0) [root,label=below left:$34$, label=below right:$\,\,\,\,\, \, \ell(p)$] {};
\draw (0,0) -- (2) -- (3) -- (7,0);
\node (15) at (3.5,2.5) [plain, label=above right:$p_4$] {};
\node (14) at (2.75,1.75) [plain, label=above right:$p_2$] {};
\node (35) at (4.5,1.5) [plain, label=above right:$p_3$] {};
\draw[style=very thin] (1) -- (14) -- (4);
\draw[style=very thin] (3) -- (35) -- (5);
\draw[style=very thin] (14) -- (15) -- (35);
\end{scope}
\end{tikzpicture}
\qquad 
\longmapsto
\qquad 34|2|35|1|24
\]
\caption{An $E$-tuple of points $p=(p_1, \ldots, p_5)$ in the plane, their vertical and horizontal projections onto the supporting line $\ell(p)$. The corresponding bisequence is $\B(p) = 34|2|35|1|24$.}\label{fig:config}
\end{figure}

\begin{definition}\label{def:fan}
The \emph{bipermutohedral fan} $\Sigma_{E,E}$ is the configuration space of $E$-tuples of points in the real plane modulo simultaneous translation, stratified according to their bisequence.
\end{definition}

By letting the $i$th point in $p=(p_1, \ldots, p_n)$ have coordinates $p_i = (z_i, w_i)$, we may regard $p$ as a point in $\N_n \times \N_n$. Then if is proved in \cite{ArdilaDenhamHuh1} that  the bipermutahedral fan can be described alternatively as follows.

\begin{proposition} \label{prop:fan}
The bipermutahedral fan is the complete simplicial fan in $\N_n \times \N_n$ whose maximal cones are \begin{eqnarray*}
\sigma_{\B} &:=& \{(z,w) \in \N_n \times \N_n \, : \, \text{the numbers } z_1-z_k, \ldots, z_n-z_k, w_k-w_1, \ldots, w_k-w_n \\
&& \qquad  \qquad \qquad \qquad \quad \text{are weakly in the opposite order of the letters $1, \ldots, n, \1, \ldots, \n$ in $\B$\}}
\end{eqnarray*}
for each bipermutation $\B$ of $[n]$, where $k=k(\B)$ is the element appearing once in $\B$.
\end{proposition}

For example, the maximal cone of the bipermutahedral fan $\Sigma_{4,4}$ corresponding to the bipermutation $2|3|4|2|4|1|1$ -- which we rewrite as $2|{\bf 3}|4|\2|\4|1|\1$ -- is given by the following inequalities
\[
\sigma_{2|3|4|2|4|1|1}: \qquad 
z_2-z_3 \geq
0 \geq
z_4-z_3 \geq
w_3-w_2 \geq
w_3-w_4 \geq
z_1-z_3 \geq
w_3-w_1.
\]

\subsection{Constructing the bipermutahedron}\label{sec:construction}

We now recall the construction of the bipermutahedron $\Pi_{n,n}$ from \cite{ArdilaDenhamHuh1}. For each bipermutation $\B$, we construct a vertex $v_\B$ in $\M_n \times \M_n$ as follows.

First, let $k=k(\B)$ be the element appearing only once in $\B$, and consider the word obtained by replacing the first and second occurrences of each $i \neq k$ with $i$ and $\ii$ respectively, and replacing $k$ with $k\overline{k}$. Then identify this word
with a bijection $\pi=\pi(\B)$:
\[
\pi(\B) 
\colon 
E \cup \E
\longrightarrow
\{-(2n-1), -(2n-3), \ldots, -3, -1, 1, 3, \ldots, (2n-3), (2n-1)\}
\]
that sends the letters of the word to $-(2n-1), \ldots, -1, 1, \ldots, (2n-1)$ in increasing order. For example, 
the bipermutation $2|3|4|2|4|1|1$ is sent to the bijection
\[
2|3|4|2|4|1|1 \longmapsto
23\34\2\41\1 \longmapsto
\pi = 
\left(
\begin{array}{rrrrrrrr}
2 & 3 & \3 & 4 & \2 & \4 & 1 & \1 \\
-7 & -5 & -3 & -1 & \,\,\,1 & \,\,\,3 & \,\,\,5 & \,\,\,7 \\
\end{array} \right)
\]
with $\pi(2) = -7, \pi(3) = -5, \ldots, \pi(\1) = 7$. 

Next, to the bijection $\pi$ we associate a vector $u_\pi =(x,y) \in \R^E \times \R^E$ with coordinates
  $x_i = \pi(i)$ and $y_i = -\pi(\ii)$ for $i \in E$.
%
Notice that $u_\pi$ is on the hyperplane $\sum_{i \in E} x_i - \sum_{i \in E} y_i = 0$, so we may define the number $s_\pi = \sum_{i \in E} x_i = \sum_{i \in E} y_i$.
Writing vectors $(x,y) \in \R^n \times \R^n$ in a $2 \times n$ table whose top and bottom rows are $x$ and $y$ respectively, we have, for example,
\[
u_{23\34\2\41\1} = 
\begin{array}{|rrrr|}
\hline
5 & -7 & -5 & -1 \\
-7 & -1 & 3 & -3 \\
\hline
\end{array} \, ,
\qquad
s_\pi = -8.
\]
Finally define the vertex
\[
v_\BB = u_{\pi(\BB)} - s_{\pi(\BB)}(\e^k + \f^k).
\]
For example,
\begin{eqnarray*}
v_{2|3|4|2|4|1|1} &=& u_{23\34\2\41\1} - s_{23\34\2\41\1}(\e^2+\f^2) \\ 
&=& 
\begin{array}{|rrrr|}
\hline
5 & -7 & -5 & -1 \\
-7 & -1 & 3 & -3 \\
\hline
\end{array} 
+ 8 \, \, 
\begin{array}{|rrrr|}
\hline
0 & 0 & 1 & 0 \\
0 & 0 & 1 & 0 \\
\hline
\end{array} 
= 
\begin{array}{|rrrr|}
\hline
5 & -7 & 3 & -1 \\
-7 & -1 & 11 & -3 \\
\hline
\end{array}
\end{eqnarray*}
The row sums of $v_\BB$ equal $0$, so $v_\BB \in \M_n \times \M_n$.

\begin{definition} \label{def:biperm}
The \emph{bipermutahedron} on $[n]$ is 
\[
\Pi_{n,n} := \conv \{v_\BB \, : \, \BB \text{ is a bipermutation on } [n]\} \subset  \M_n \times \M_n.
\]
\end{definition}

Recall that the \emph{(inner) normal fan} $\mathcal{N}(P)$ of a polytope $P$ in a vector space $V$ is the complete fan in the dual space $V^*$ whose maximal cones are
\[
\sigma_v = \{w \in V^* \, : \, w(v) \leq w(x) \text{ for all } x \in P\}
\]
for the vertices $v$ of $P$. The face poset of $\mathcal{N}(P)$ is anti-isomorphic to the face poset of $P$.

\begin{theorem} \cite{ArdilaDenhamHuh1}
The bipermutahedral fan is the normal fan of the bipermutahedron.
\end{theorem}

\subsection{The face structure of the bipermutahedron.}

\begin{definition}\label{def:Bseq}
A \emph{bisequence} on $E$ is a sequence $\B=B_1|\cdots|B_m$ of nonempty subsets of $E$, called the \emph{parts} of $\B$, such that
\begin{enumerate}[(1)]\itemsep 5pt
\item  every element of $E$ appears in at least one part of  $\B$, 
\item  every element of $E$ appears in at most two parts of  $\B$, and 
\item some element of $E$ appears in exactly one part of $\B$.
\end{enumerate}
A \emph{bisubset} of $E$ is a bisequence of length $2$.
A \emph{bipermutation} of $E$ is a bisequence  of length $2n-1$. 
The \emph{poset of bisequences} $\B_n$ consists of the bisequences on $[n]$ ordered by adjacent refinement, so $\B \leq \B'$ if $\B$ can be obtained from $\B'$ by merging adjacent parts.
\end{definition}

For example
$23|124 \leq 23|24|1 \leq 2|3|4|2|4|1|1$ in the poset $\B_4$.
 The poset of bisequences on $E$ is a graded poset.
 Its $k$-th level  consists of the bisequences of $k+1$ nonempty subsets of $E$,  and the top level consists of the bipermutations of $E$.

%
%

\begin{proposition} \label{prop:faces} \cite{ArdilaDenhamHuh1} The face poset of the bipermutahedron $\Pi_{n,n}$ is anti-isomorphic to the poset of bisequences $\B_n$; that is:
\begin{enumerate}
\item
The faces of the bipermutahedron are in bijection with the bisequences on $[n]$.

\item 
The dimension of the face labeled by $\B$ is one less than the number of parts of $\B$.

\item
Two faces $F$ and $F'$ of the bipermutahedron satisfy $F \supseteq F'$ if and only if their bisequences satisfy $\B \leq \B'$ in $\B_n$.
\end{enumerate}
\end{proposition}

Figure \ref{fig:biperm} shows the bipermutahedron $\Pi_{2,2}$ and the bipermutahedral fan $\Sigma_{2,2}$, with its faces labeled by the  bisequences on $\{1,2\}$.

\begin{figure}[h]
\[
	\begin{tikzpicture}[scale=1]
root/.style={circle,draw,inner sep=1.2pt,fill=black}]
	\filldraw[fill=cyan,very thick] 
(0,0) node[very thick] {$\bullet$} node [below] {$\begin{array}{|rr|}\hline  -3 & \,\,\,\,\,3 \\ -3 & 3 \\ \hline \end{array}$}
--
(1,1) node {$\bullet$} node [right] {$\begin{array}{|rr|}\hline  -3 & 3 \\ 1 & -1 \\ \hline \end{array}$}
--
(1,3) node {$\bullet$} node [right] {$\begin{array}{|rr|}\hline  -1 & 1 \\ 3 & -3 \\ \hline \end{array}$}
--
(0,4) node {$\bullet$} node [above] {$\begin{array}{|rr|}\hline  \,\,\,\,\, 3 & -3 \\ 3 & -3 \\ \hline \end{array}$}
--
(-1,3) node {$\bullet$} node [left] {$\begin{array}{|rr|}\hline  3 & -3 \\ -1 & 1 \\ \hline \end{array}$}
--
(-1,1) node {$\bullet$} node [left] {$\begin{array}{|rr|}\hline  1 & -1 \\ -3 & 3 \\ \hline \end{array}$}
--
(0,0);

 \begin{scope}[shift={(8,2)}, scale=.9, root/.style={circle,draw,inner sep=1.2pt,fill=black},every node/.style={scale=0.95}]
\draw[style=very thick,color=gray] (2,0) -- (-2,0);
\draw[style=very thick,color=gray] (-2,-2) -- (2,2);
\draw[style=very thick,color=gray] (2, -2) -- (-2, 2);
\node at (0,-1.5) {\footnotesize$2|1|2$};
\node at (0,1.5) {\footnotesize$1|2|1$};
\node at (1.5,-0.7) {\footnotesize$2|2|1$};
\node at (-1.5,-0.7) {\footnotesize$1|2|2$};
\node at (1.5,0.7) {\footnotesize$2|1|1$};
\node at (-1.5,0.7) {\footnotesize$1|1|2$};
\node at (0,0) [root, label={[label distance=0pt]below:{\grey{\footnotesize{$12$}}}}] {};
\node at (1.8,-1.45) [label={[label distance=5pt]300:{\grey{\small{$2|12$}}}}] {};
\node at (1.8,1.45) [label={[label distance=5pt]60:{\grey{\small{$12|1$}}}}] {};
\node at (-1.8,-1.45) [label={[label distance=5pt]240:{\grey{\small{$12|2$}}}}] {};
\node at (-1.8,1.45) [label={[label distance=5pt]120:{\grey{\small{$1|12$}}}}] {};
\node at (1.9,0) [label=right:{\grey{\small{$2|1$}}}] {};
\node at (-1.9,0) [label=left:{\grey{\small{$1|2$}}}] {};
\node (1) at (3,0) {};
\node (2) at (4.5,0) {};
 \end{scope}
\end{tikzpicture}
\]
\caption{\label{fig:biperm} The bipermutahedron $\Pi_{2,2}$ and its normal fan, the bipermutahedral fan $\Sigma_{2,2}$.}
\end{figure}
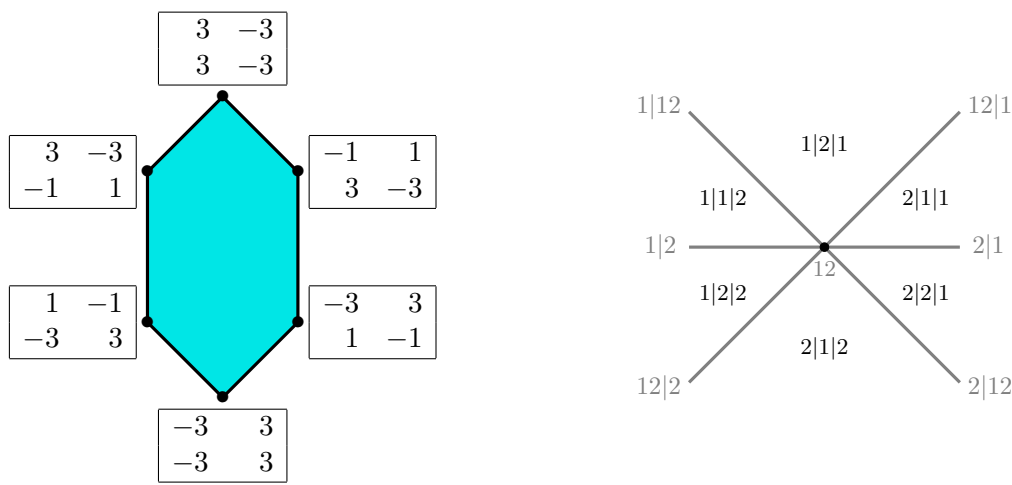

Since the bipermutahedral fan is simplicial \cite{ArdilaDenhamHuh1}, the bipermutahedron is simple.

\begin{proposition}\label{prop:inequalities}
The bipermutahedron $\Pi_{n,n}$ is given by the following minimal inequality description in $\R^n \times \R^n$:
\begin{eqnarray*}
\sum_{e \in [n]} x_e &=& 0, \\
\sum_{e \in [n]} y_e &=& 0, \\
\sum_{s \in S} x_s + \sum_{t \in T} y_t &\geq& - \, \big(|S|+|S-T|\big) \cdot \big(|T|+|T-S|\big) \quad \text{ for each bisubset $S|T$ of $[n]$}.
\end{eqnarray*}
\end{proposition}

\begin{proof}
The first two equations hold, and determine a codimension two subspace perpendicular to the lineality space $\R\{\e_E, \f_E\}$ of $\NN(\Pi_{n,n})$. The minimal inequality description is then determined by the rays $\e_S + \f_T$ for the bisubsets $S|T$, and each inequality is essential since the facets of the bipermutahedron are in bijection with the bisubsets of $[n]$.

Consider a bisubset $S|T$. The linear functional $\e_S+\f_T$ is minimized for the facet of $\Pi_{n,n}$ labeled by the bisequence $S|T$, and hence for any vertex $v_\B$ indexed by a subsequence $\B$ refining $S|T$. Consider such a bisequence $\B$ and let $k$ be its non-repeated element. 
Since $k$ appears only once in $\B$, it only appears once in $S|T$, so $(\e_S+\f_T)(\e^k + \f^k)=1$. Thus
\begin{eqnarray*}
(\e_S+\f_T)(v_\B) &=& (\e_S+\f_T)(u_{\pi(\B)} - s_{\pi(\B)}(\e^k + \f^k)) \\
&=& \sum_{s \in S} \pi(s) + \sum_{t \in T} (-\pi(\bt)) - \sum_{t \in [n]} (-\pi(\bt)) \\
&=& \sum_{s \in S} \pi(s) + \sum_{t \in [n]-T} \pi(\bt) \\
&=& \sum_{s \in S} \pi(s) + \sum_{t \in S-T} \pi(\bt)
\end{eqnarray*}
is the sum of the values of the function $\pi$ on $S$ and $\overline{S-T}$. To compute this sum, notice that $S|T = [(S-T) \cup (S \cap T)] \, |  \,[(S \cap T) \cup (T-S)]$, so for any bipermutation $\B$ refining $S|T$, the word $\pi(\B)$ must contain the numbers $(S-T) \cup (\overline{S-T})  \cup (S \cap T)  = S \cup (\overline{S-T}) $ in the first $r$ positions and the numbers $(\overline{S \cap T}) \cup (T-S) \cup (\overline{T-S}) = (T-S) \cup \overline{T}$ in the last $2n-r$ positions, where $r=|S| + |S-T|$ and $2n-r = |T| + |T-S|$. It follows that
\begin{eqnarray*}
\sum_{s \in S} \pi(S) + \sum_{t \in S-T} \pi(\bt) &=& -(2n-1) -(2n-3) - \cdots - (2n-2r+1) \\
&=& r(-2n+r).
\end{eqnarray*}
This completes the proof.
\end{proof}

The original construction of the bipermutahedron $\Pi_{n,n}$, given in Section \ref{sec:construction}, may seem somewhat complicated on first sight. However, its inequality description is remarkably simple, and very reminiscent of that of the  standard permutahedron, which is given by 
\begin{eqnarray*}
\sum_{e \in [n]} x_e &=& 0, \\
\sum_{s \in S} x_s &\geq& -\, |S| \cdot |E-S|  \qquad \text{for each subset $\emptyset \subsetneq S \subsetneq [n]$}.
\end{eqnarray*}
This makes one suspect that this is one of nicest polytopes whose normal fan is the bipermutahedral fan $\Sigma_{n,n}$. We prove a precise statement to this effect in Proposition \ref{prop:symmetry}: the bipermutahedron has the largest possible symmetry group.

Readers familiar with deformations of permutahedra, as studied by Postnikov \cite{Postnikov}, may wonder whether bipermutahedra belong to this family of polytopes; on the surface, they look like they do.
However, the bipermutahedron is \textbf{not} a deformation of a permutahedron. 
One way to see this is to observe that the bipermutahedral fan has walls spanning hyperplanes of the form $x_i+y_i=x_j+y_j$, which are not in the braid arrangement. 

We know that deformations of permutahedra are in bijection with submodular functions \cite{Edmonds, Fujishige}. In Section \ref{sec:amplecone} we give an analogous description of the cone of deformations of bipermutahedra.

\section{The $f$-vector}\label{sec:f-vector}

In this section we compute the $f$-vector of the bipermutahedron. The formulas are slightly simpler for the reverse sequence, the $f$-vector of the bipermutahedral fan.
Recall that the \emph{$f$-vector} of a $d$-dimensional fan $\Delta$ is  $f_\Delta=(f_0, \ldots, f_d)$ where $f_i$ is the number of $i$-dimensional faces of $\Delta$.

\begin{proposition}\label{prop:fvector}
The $f$-vector of the bipermutahedral fan $\Sigma_{n,n}$ is given by
\begin{eqnarray*}
f_{d-2}(\Sigma_{n,n}) &=& \text{\# of multigraphs on vertex set $[d]$ and edge set $[n]$ and no isolated vertices} \\
&=& \sum_{i=0}^d(-1)^{d-i}{d \choose i}{i \choose 2}^n
\end{eqnarray*}
for $2 \leq d \leq 2n$.
\end{proposition}

\begin{proof}
Each  $(d-2)$-dimensional faces of $\Sigma_{n,n}$ is indexed by a bisequence $\B$ of $[n]$ with $d-1$ parts. We can use it to construct an ordered set partition $o(\B)$ of $\{1,1,2,2,\ldots,n,n\}$ with $d$ parts by adding a final part to $\B$ consisting of the elements of $[n]$ that only appear once in $\B$. The ordered set partitions that arise are those where no part contains repeated elements. 

In turn, we can use $o(\B)$ to construct a multigraph $G(\B)$ on vertex set $[d]$ and edge set $[n]$ by letting edge $e$ connect the vertices $i$ and $j$ such that $e$ is in the $i$th and $j$th parts of $\B$. The multigraphs that arise are those with no isolated vertices.
An example for a 4-dimensional face of $\Sigma_{5,5}$ is shown below. 
\[
\B = 1|1\mathbf{4}|35|35|\mathbf{2}
\quad \longmapsto \quad 
o(\B) = 1|14|35|35|2|24 
\quad \longmapsto \quad 
G(\B) = 
\begin{tikzpicture}[scale=1.25,baseline=(current bounding box.center),
plain/.style={circle,draw,inner sep=1.2pt,fill=white},
root/.style={circle,draw,inner sep=1.2pt,fill=black}]
\node (1) at (60:1) [root, label=above right:${\bf 1}$] {};
\node (2) at (0:1) [root, label=right:${\bf 2}$] {};
\node (3) at (300:1) [root, label=below right:${\bf 3}$] {};
\node (4) at (240:1) [root, label=below left:${\bf 4}$] {};
\node (5) at (180:1) [root, label=left:${\bf 5}$] {};
\node (6) at (120:1) [root, label=above left:${\bf 6}$] {};
\draw (1) -- (2);
\draw (30:1.1) node {1};
\draw (5) -- (6);
\draw (150:1.1) node {2};
\draw (3) to[out=150,in=30] (4);
\draw (270:0.5) node {3};
\draw (3) to[out=-150,in=-30]  (4);
\draw (90:0.4) node {4};
\draw (2) --  (6);
\draw (270:1.2) node {5};
\end{tikzpicture}
\]

It is straightforward to recover the bisequence $\B$ from its associated multigraph $G(\B)$. This proves the first claim.

We can then use the inclusion-exclusion formula to compute
\begin{eqnarray*}
f_{d-2}(\Sigma_{n,n}) &=& \text{(multigraphs on vertex set $[d]$ and edge set $[n]$ and no isolated vertices)} \\
&=& \sum_{S \subseteq [d]} (-1)^{|S|} \text{(multigraphs on $V=[d]$, $E=[n]$ where each vertex in $S$ is isolated)} \\
&=& \sum_{S \subseteq [d]} (-1)^{|S|} \text{(multigraphs on $V=[d]-S$, \, $E=[n]$)} \\
&=& \sum_{S \subseteq [d]} (-1)^{|S|} {d-|S| \choose 2}^n,
\end{eqnarray*}
giving the desired result. \end{proof}

We can give a more concise expression in terms of Sokal's \emph{deformed exponential function} of \cite{Sokal}, which is the following evaluation of the three variable Rogers-Ramanujan function:
\[
F(\alpha, \beta) = \sum_{n \geq 0} \frac{\alpha^n \, \beta^{n \choose 2}}{n!}.
\]
This function has been widely studied in complex analysis \cite{Langley, Liu, Morris} and statistical mechanics \cite{SS, SS2, Sokal}. It also arises naturally in the computation of the Tutte polynomial and arithmetic Tutte polynomials of the classical root systems \cite{Ardila, ArdilaCastilloHenley}.

\begin{theorem}\label{th:fseries}
The double exponential generating function for the face numbers of the bipermutahedral fans is
\[
\sum_{n \geq 0} \sum_{d \geq 2} f_{d-2}(\Sigma_{n,n}) \frac{x^d}{d!} \frac{y^n}{n!} = \frac{F(x,e^y)}{e^x}
\]
\end{theorem}

\begin{proof}
Let
\begin{eqnarray*}
g_{d,n} &=& \text{\# of multigraphs on } V = [d], E = [n], \\
c_{d,n} &=& \text{\# of multigraphs on } V = [d], E = [n] \text{ that are connected}, \\
i_{d,n} &=& \text{\# of multigraphs on } V = [d], E = [n] \text{ that have no isolated vertices},
\end{eqnarray*}
and
\[
G(x,y) = \sum_{d,n \geq 0} g_{d,n} \frac{x^d}{d!} \frac{y^n}{n!}, \qquad
C(x,y) = \sum_{d,n \geq 0} c_{d,n} \frac{x^d}{d!} \frac{y^n}{n!}, \qquad
I(x,y) = \sum_{d,n \geq 0} i_{d,n} \frac{x^d}{d!} \frac{y^n}{n!} 
\]
be their double exponential generating functions. 
The Exponential Formula for exponential generating functions \cite[Corollary 5.1.6]{EC2} gives
\[
G(x,y) = e^{C(x,y)},  \qquad
I(x,y) = e^{C(x,y)-x},  \qquad
\]
since $x$ is the generating function for the graph with one isolated vertex. It follows that
\[
I(x,y) = \frac{G(x,y)}{e^x}.
\]
It remains to compute
\begin{eqnarray*}
G(x,y) &=&  \sum_{n,d \geq 0} {d \choose 2}^n \frac{x^d}{d!} \frac{y^n}{n!}  \\
&=& \sum_{d \geq 0} e^{{d \choose 2}y} \frac{x^d}{d!} \\
&=& F(x,e^y),
\end{eqnarray*}
implying the desired result.
\end{proof}

Using Proposition \ref{prop:fvector} or \ref{th:fseries} one easily computes the $f$-vector of the first few bipermutahedra:
\begin{eqnarray*}
f(\Pi_{1,1}) &=& (1,1), \\
f(\Pi_{2,2}) &=& (1,6,6,1), \\
f(\Pi_{3,3}) &=& (1, 90, 180, 114, 24, 1), \\
f(\Pi_{4,4}) &=& (1, 2520, 7560, 8460, 4320, 978,78,1).
\end{eqnarray*}

\section{The $h$-vector and the biEulerian polynomial}\label{sec:descents}

Recall that the \emph{$h$-vector} of a $d$-dimensional simplicial fan $\Delta$ is  $h_\Delta=(h_0, \ldots, h_d)$ where
\[
h_0(x+1)^d + \cdots + h_d(x+1)^0 = 
f_0 x^d + \cdots + f_d x^0,
\]
where $f_\Delta=(f_0, \ldots, f_d)$ is the $f$-vector of the fan $\Delta$. This is a more economical encoding of the $f$-vector, because the Dehn-Somerville relations guarantee that $h_i = h_{d-i}$ for all  $i$. The $h$-vector is also geometrically significant: if $\Delta$ is an integral fan, then the Hilbert polynomial of the corresponding toric variety $X_{\Delta}$ is the $h$-polynomial $h_dx^d + \cdots + h_0x^0$.
We now give a combinatorial interpretation of the $h$-vector of the bipermutahedral fan.

\begin{definition} \label{def:descent}
Let $\B$ be a bipermutation on $[n]$ whose non-repeated element is $k(\B) = k$. Two consecutive elements $i|j$ of $\B$ form a \emph{descent} if one of the following conditions holds:
%

\bgroup
\def\arraystretch{1.5}
\begin{tabular}{lcl}
a) $i, j \in E$ and $i>j$, & & c) $i \in E-k$, $j \in \overline{E-k}$ and $i > k$, \\
b) $i, j \in \overline{E}$ and $i < j$, & & d) $i \in \overline{E-k}$, $j \in E-k$ and $j < k$,
\end{tabular}
\egroup

\noindent where we interpret the non-repeated element $k$ of $\B$ as being $k$ (resp. $\overline{k}$) when it is compared with an element of $E$ (resp. of $\overline{E}$).
Otherwise, $i|j$ form an \emph{ascent} of $\B$.
Let $\des(\B)$ and $\asc(\B)$ denote the number of descents and ascents of $\B$, respectively.
\end{definition}

For example the bipermutation $\B=5|4|2|3|1|4|1|2|5=5|4|\2|{\bf 3}|1|\4|\1|2|\5$, where $k=3$, has five descents:  
$5|4$ (of type a), 
$4|\2$ (of type c), 
$\2|\bf{3}$ (of type b),
${\bf 3}|1$ (of type a), and
$\1|2$ (of type d).

\begin{definition} \label{def:biEulerian}
The $n$th \emph{biEulerian polynomial} $\displaystyle B_n(x) = \sum_{i=0}^{2n-2} b(n,i)x^i$ is given by
\[
b(n,i) = \text{ number of bipermutations of $[n]$ with $i$ descents, for } 0 \leq i \leq 2n-2.
\]
\end{definition}

The first few biEulerian polynomials are the following.
\begin{eqnarray*}
B_1(x) &=& 1, \\
B_2(x) &=& 1+4x+x^2, \\
B_3(x) &=& 1+20x+48x^2+20x^3+x^4, \\
B_4(x) &=& 1+72x+603x^2+1168x^3+603x^4+72x^5+x^6.
\end{eqnarray*}

%

\noindent
The biEulerian polynomial is analogous to the \emph{Eulerian polynomial} $A_n(x)$, which enumerates the permutations of $[n]$ according to their number of descents. The following result is analogous to the fact that the $h$-polynomial of the permutahedral fan is the Eulerian polynomial.

\begin{theorem} \label{th:biEulerian}
The $h$-polynomial of the bipermutahedral fan $\Sigma_{n,n}$ is the $n$th biEulerian polynomial.
\end{theorem}

\begin{proof}
If $P$ is a simple polytope, then the $h$-vector of its normal fan can be computed in terms of a sweep hyperplane on $P$, as follows \cite[Section 9.9]{Grunbaum}.
For a generic linear functional $\lambda$, if we orient the edges of $P$ in the direction $u \rightarrow v$ for $\lambda(u) > \lambda(v)$, then 
\[
h_i (P) = \text{ number of vertices of $P$ with indegree } i.
\]
Let us choose a generic linear functional $\lambda = (z, w) \in \N_n \times \N_n$ with 
\[
z_n  >> \cdots >> z_2  >> z_1 >> w_n > \cdots > w_2 > w_1 > 0.
\]
We claim that for the vertex $v=v_B$ corresponding to bipermutation $\B$, we have
\begin{equation}\label{eq:indeg}
\text{indegree}(v) = \text{number of descents of } \B.
\end{equation}
The desired result will follow.

Let $\B= b_1| \cdots | b_{2n-1}$ with non-repeated element $k$ and let $v$ be the corresponding vertex.
By Proposition \ref{prop:faces}, the $2n-2$ edges $vv_h$ containing the vertex $v$ of $\B$ correspond to the $2n-2$ pairs of adjacent elements, say $b_h|b_{h+1} = i|j$, for $1 \leq h \leq 2n-2$. Identifying a vertex $v_h$ and its bipermutation $\B_h$, we have the following.
%

\medskip
\noindent
A) If $j \neq i$, the $h$th neighbor of $\B(v) = b_1| \cdots | i | j | \cdots | b_{2n-1}$ is $\B(v_h) = b_1| \cdots | j | i | \cdots | b_{2n-1}$. 

\medskip
\noindent
B) If $j = i$, the $h$th neighbor of $\B(v) = b_1| \cdots | i | i | \cdots | {\bf k} | \cdots | b_{2n-1}$ is $\B(v_h) = b_1| \cdots | {\bf i} | \cdots | k | k | \cdots | b_{2n-1}$. 

\medskip

\noindent
We will now prove that the edge $vv_h$ has direction $v \leftarrow v_h$ if and only if $i|j$ is a descent of $\B$.
First consider case A). It has eight subcases.

\medskip

\noindent
\emph{Case a)}: $i, j \in E$. The permutations $\pi:=\pi(\B)$ and $\pi_h:=\pi(\B_h)$ only differ in the positions $\pi(i)=a-1,  \pi(j) = a+1$ and $\pi'(i)=a+1, \, \pi'(j) = a-1$ for some $a$. Thus $s(\pi') = s(\pi) =:s$, and also $k(\B) = k(\B_h) = k$. Then we have
\begin{eqnarray*}
v-v_h &=& (u - s(\e^k+\f^k)) - (u_h - s(\e^k+\f^k)) \\
&=&
\begin{blockarray}{ccccc}
 & i & & j &  \\
\begin{block}{|ccccc|}
\hline
\cdot & a-1 & \cdot  &  a+1 & \cdot \\
\cdot & \cdot & \cdot & \cdot & \cdot \\
\hline
\end{block}
\end{blockarray}
\,
-
\begin{blockarray}{ccccc}
 & i & & j &  \\
\begin{block}{|ccccc|}
\hline
\cdot & a+1 & \cdot  &  a-1 & \cdot \\
\cdot & \cdot & \cdot & \cdot & \cdot \\
\hline
\end{block}
\end{blockarray}
\\
&=& 
2(\e^j-\e^i),
\end{eqnarray*}
%
so
\[
(z, w)(v-v_h) = 2(z_j-z_i)
\]
and the edge $vv_h$ is directed $v \leftarrow v_h$ if and only if this quantity is negative. This happens precisely when $i>j$, which agrees with case a) of Definition \ref{def:descent} of a descent.

\medskip

\noindent \emph{Case b):} $i, j \in \overline{E}$. Similarly to Case a), $v-v_h= 2(\f^i-\f^j)$ 
so 
$
(z, w)(v-v_h) = 2(w_i - w_j)
$
and the edge $vv_h$ is directed $v \leftarrow v_h$ if and only if $i<j$. This agrees with case b) of Definition \ref{def:descent}.

\medskip

\noindent \emph{Case c):} $i \in E-k$ and $j \in \overline{E-k}$. Now $\pi$ and $\pi'$ , $\pi(i)=a-1,  \pi(j) = a+1$ and $\pi_h(i)=a+1, \, \pi_h(j) = a-1$ for some $a$. Now we have that $s(\pi_h) = s(\pi)+2$, so 
\begin{eqnarray*}
v-v_h &=& (u - s(\e^k+\f^k)) - (u_h - (s+2)(\e^k+\f^k)) \\
&=&
\begin{blockarray}{ccccc}
 & i & & j &  \\
\begin{block}{|ccccc|}
\hline
\cdot & a-1 & \cdot  &  \cdot & \cdot \\
\cdot & \cdot & \cdot & -(a+1) & \cdot \\
\hline
\end{block}
\end{blockarray}
\,
-
\begin{blockarray}{ccccc}
 & i & & j &  \\
\begin{block}{|ccccc|}
\hline
\cdot & a+1 & \cdot  &  \cdot & \cdot \\
\cdot & \cdot & \cdot & -(a-1) & \cdot \\
\hline
\end{block}
\end{blockarray} + 2(\e^k+\f^k)
\\
&=& 
2(-\e^i - \f^j) + 2(\e^k+\f^k),
\end{eqnarray*}
so
\[
(z, w)(v-v_h) = 2(z_k-z_i) + 2(w_k - w_j) 
\]
and the edge $vv_h$ is directed $v \leftarrow v_h$ when $i>k$, in agreement with case c) of Definition \ref{def:descent}.

\medskip

\noindent \emph{Case d):} $i \in \overline{E-k}$ and $j \in E-k$. 
Similarly to Case c),  
$(z, w)(v-v_h) = 2(z_j-z_k) + 2(w_i - w_k)$ 
so the edge $vv_h$ is directed $v \leftarrow v_h$ if and only if $j<k$, in agreement with case d) of Definition \ref{def:descent}.

\medskip

\medskip
 
Now we consider case B). In this case $v$ and $v_h$ are assigned the same permutation $\pi = \pi_h$, so $v-v_h = -s(\e^k+\f^k-\e^i-\f^i)$
and
\[
(z, w)(v-v_h) = s\big((z_i - z_k) + (w_i - w_k)\big).
\]
Notice that, since $j$ precedes $\overline{j}$ in $\pi$ for all $j$, we have
\[
s \leq -(2n-1) - (2n-5) - (2n-9) - \cdots + (2n-7) + (2n-3) < 0.
\]
It follows that the edge $vv_h$ is directed $v \leftarrow v_h$ if and only if $i>k$. This is consistent with case c) of Definition \ref{def:descent}.

\medskip

In summary, the edges pointing towards $v$ are in bijection with the descents of the corresponding bipermutation $\B$. The desired result follows.
\end{proof}

The $h$-vector of any simple polytope is symmetric thanks to the Dehn-Sommerville relations \cite{Ziegler}. For the bipermutahedron there is a simple combinatorial explanation, in light of Theorem \ref{th:biEulerian}. If $\rev(\B)$ is the reverse of the bipermutation $\B$, then the descents of $\B$ become ascents in $\rev(\B)$ and viceversa. Therefore $\des(\rev(\B)) = 2n-2-\des(\B)$. This implies that $h_i(\Pi_{n,n}) = h_{2n-2-i}(\Pi_{n,n})$.

\begin{corollary}
The Hilbert polynomial of the bipermutahedral variety $X_{n,n}$ is the biEulerian polynomial.
\end{corollary}

%

\section{The bipermutahedral triangulation of the product of triangles}

In this section we construct a unimodular triangulation of the product of $n$ triangles that is intimately tied to the bipermutahedron. This will be used in the next section to find a formula for the biEulerian polynomial, using the Ehrhart theory of this triangulation.

Let $\e,\f,\g$ be the standard basis of $\R^3$, and consider the standard equilateral triangle and its $n$-fold product in $\R^{3 \times n}$:
\[
\Delta := \conv(\e,\f,\g), \qquad \Delta^n := \underbrace{\Delta \times \cdots \Delta}_{n \text{ times}}.
\]
We identify $\R^{3n}$ with $\R^n \times \R^n \times \R^n$ with standard bases $\{\e_1, \ldots \e_n\}, \{\f_1, \ldots, \f_n\}, \{\g_1, \ldots, \g_n\}$, respectively. We write vectors in $\R^{3 \times n}$ as $3 \times n$ tables.

The polytope $\Delta^n$ has $3^n$ vertices, namely, the $0-1$ tables of shape $3 \times n$ having exactly one $1$ in each column. These vertices are in bijection with the $3^n$ pairs $(S,T)$ of subsets of $[n]$ whose union is $[n]$ as follows:
\[
v_{S,T}: = 
\begin{bmatrix}
&&& \e_{E-S} &&&  \\
&&& \f_{E-T} &&&  \\
&&& \g_{S \cap T} &&& 
\end{bmatrix}
\qquad
 \text{ for } \emptyset \subseteq S, T \subseteq [n], \,\, S \cup T = [n].
\]
For example
\[
v_{235,1345} =
\begin{bmatrix}
1 & 0 & 0 & 1 & 0 \\
0 & 1 & 0 & 0 & 0 \\
0 & 0 & 1 & 0 & 1
\end{bmatrix}.
\]

\begin{theorem}\label{th:triangulation}
There is a unimodular triangulation of $\Delta^n$ that is combinatorially isomorphic to a triple cone over the bipermutahedral fan $\Sigma_{n,n}$.
\end{theorem}

\begin{proof}
Consider the composite map

\[
\begin{tabular}{ccccccc}
$\pi$  & : & $\R^{3 \times n}$ & $\xrightarrow[]{\,\,\pi_1\,\,}$ & $\R^{2 \times n}$ & $\xrightarrow[]{\,\,\pi_2\,\,}$ &  $\N_n \times \N_n$ \\
      &   &  
$
\begin{bmatrix}
&&& \uu &&&  \\
&&& \vv &&&  \\
&&& \ww &&& 
\end{bmatrix}
$
& $\longmapsto$  & 
$\begin{bmatrix}
&& \one-\uu &&  \\
&& \one-\vv &&  
\end{bmatrix}$
& $\longmapsto$  & 
$
\begin{bmatrix}
&& \overline{\one-\uu} &&  \\
&& \overline{\one-\vv} &&  
\end{bmatrix}
$
\end{tabular}
\]
Notice that $\pi$ maps the $3^n-3$ vertices corresponding to the bisubsets $S|T$ of $[n]$ to the $3^n-3$ rays of the bipermutahedral fan:
\[
\pi \,:\, v_{S,T} \,\xrightarrow[]{\,\,\pi_1\,\,}\, \e_S + \f_T \,\xrightarrow[]{\,\,\pi_2\,\,}\, \e_{S|T},
\]
whereas $\pi(v_{\emptyset, E}) = \pi(v_{E,\emptyset}) = \pi(v_{E, E}) = 0$.
Also,  the polytope $\Delta^n$ lives in the $2n$-dimensional affine plane 
\[
P = \{(\uu,\vv,\ww) \in \R^{3 \times n} \, : \, u_1+v_1+w_1 = \cdots = u_n+v_n+w_n=1\} \subset \R^{3 \times n}
\]
and the restriction $\pi_1 : P \rightarrow \R^{2 \times n}$ is a unimodular bijective transformation, which induces a bijection between $\{v_{S|T} \,: \, S|T \text{ is a bisubset of } [n]\}$ and $\{\e_S+\f_T : \, S|T \text{ is a bisubset of } [n]\}$.

\medskip

For each bipermutation $\B$ of $[n]$,
consider the $(2n)$-simplex $T_{\B} \subset \Delta^n$ given by
\[
T_{\B} := \conv\{\vv_{\emptyset,E}, \vv_{E, \emptyset}, \vv_{E,E}, \vv_{S,T} \, : \, S|T \text{ is a bisubset of } \B\},
\]
where the $2n-2$ bisubsets of a bipermutation $\B=b_1|\cdots|b_{2n-1}$ are $b_1\ldots b_k | b_{k+1} \ldots b_{2n-1}$ for $1 \leq k \leq 2n-2$. For example, $T_{1|3|2|1|3} = \conv\{\vv_{\emptyset, 123}, \vv_{123, \emptyset}, \vv_{123, 123}, \vv_{1,123}, \vv_{13, 123}, \vv_{123, 13}, \vv_{123, 3}\}.$
Notice that $T_{\B}$ is indeed a simplex because an affine dependence between these points would lead, under the projection $\pi$ to a linear dependence between the rays $\e_{S|T}$ of the simplicial chamber $\sigma_{\B}$.

\medskip

1. We claim that 
\[
\T := \{T_{\B} \, : \, \B \text{ is a bipermutation of } [n]\}
\]
is a triangulation of the product of triangles $\Delta^n$. To prove it, we need to show that these simplices cover $\Delta^n$, and that they intersect face-to-face; that is, any two of them intersect in a common face.

\medskip

1a. $\T$ covers $\Delta^n$.

\medskip

Given a point $p \in \Delta^n$, consider its image $\pi(p)$ in $\N_n \times \N_n$. The image of the simplex $\Delta^n$ is $\pi(\Delta^n) = \conv\{\e_{S|T} \, : \, S|T \text{ is a bisubset of } [n]\}$. Therefore, choosing $\B$ to be a bipermutation such that $\pi(p)$ is in the chamber $\sigma_{\B} =  \cone\{\e_{S|T} \, : \, S|T \text{ is a bisubset of } \B\}$ of $\N_n \times \N_n$, we must have
\[
\pi(p) \in \conv\{0, \e_{S|T} \, : \, S|T \text{ is a bisubset of } \B\}.
\]
This means that there are scalars $\lambda_{S|T} \geq 0$ with $\sum_{S|T} \lambda_{S|T} \leq 1$ such that
\[
\pi(p)  =  \sum_{S|T \subseteq \B} \lambda_{S|T} \e_{S|T} \qquad  \text{ in } \N_n \times \N_n,
\]
using the symbol $\subseteq$ for bisubsets. Therefore there are unique scalars $\lambda$ and $\mu$ such that 
\begin{equation}\label{eq:lambdamu}
\pi_1(p)  =  \sum_{S|T \subseteq \B} \lambda_{S|T} (\e_S + \f_T) + \lambda \e_E + \mu \f_E \qquad  \text{ in } \R^{2 \times n}.
\end{equation}
so 
\[
\pi_1(p)  =  \sum_{S|T \subseteq \B} \lambda_{S|T} (\e_S + \f_T) + a(\e_{\emptyset} + \f_E) + b(\e_E + \f_{\emptyset}) + c(\e_E + \f_E) \qquad  \text{ in } \R^{2 \times n}.
\]
for any scalars $a,b,c$ with $b+c=\lambda$ and $a+c=\mu$. Recall $\pi_1$ is a bijection from $P$ to $\R^{2 \times n}$, so
\[
p  =  \sum_{S|T \subseteq \B} \lambda_{S|T} v_{S,T} + a\, v_{\emptyset,E} + b\, v_{E, \emptyset} + c\, v_{E,E} \qquad  \text{ in } \R^{3 \times n}.
\]
This expresses $p$ as a linear combination of the vertices of the simplex $T_{\B}$. To conclude that $p \in T_{\B}$, we need the coefficients of the right hand side to be non-negative and add up to $1$.
The second condition is satisfied uniquely by the choices
\[
a= 1 - \sum_{S|T \subseteq \B} \lambda_{S|T} - \lambda, \qquad
b= 1 - \sum_{S|T \subseteq \B} \lambda_{S|T} - \mu, \qquad
c= \lambda+\mu + \sum_{S|T \subseteq \B} \lambda_{S|T} - 1;
\]
we need to show that they are non-negative. To do that, we analyze \eqref{eq:lambdamu} more closely.

Let $\B=b_1|\ldots|b_{2n-1}$. Then $b_1 \in S$ for any bisubset $S|T$ of $\B$, so the coefficient of $\e_{b_1}$ in \eqref{eq:lambdamu} is
\begin{eqnarray*}
\sum_{S|T \subseteq \B} \lambda_{S|T} + \lambda
&=& [\e_{b_1}]\pi_1(p) \\
&=& 1 - [\e_{b_1}]p \\
&\leq& 1, 
\end{eqnarray*}
where the inequality follows from the fact that $p \in \Delta^n$, which lies in the positive orthant. Similarly, the coefficient of $\f_{b_{2n-1}}$ is
\begin{eqnarray*}
\sum_{S|T \subseteq \B} \lambda_{S|T} + \mu 
&=& [\f_{b_{2n-1}}]\pi_1(p) \\
&=& 1 - [\f_{b_{2n-1}}]p \\
&\leq& 1,
\end{eqnarray*}
Finally, if $k=k(\B)$ is the element that appears only once in $\B$, then for any bisubset $S|T$ of $\B$, the element $k$ appears in $S$ or in $T$ but not in both. Therefore the sum of the coefficients of $\e_k$ and $\f_k$ in \eqref{eq:lambdamu} is
\begin{eqnarray*}
\sum_{S|T \subseteq \B} \lambda_{S|T} + \lambda + \mu
&=& [\e_k]\pi_1(p) + [\f_k]\pi_1(p) \\
&=& 2 - [\e_k]p - [\f_k]p \\
&=& 1 + [\g_k]p \\
&\geq& 1.
\end{eqnarray*}
We conclude that $a,b,c \geq 0$, and $p$ is indeed covered by the simplex $T_{\B}$, as desired.

\medskip

1b. The simplices in $\T$ intersect face-to-face.

\medskip

Consider two simplices $T_{\B_1}, T_{\B_2}$ of $\T$, and let $p$ be any point in their intersection $T_{\B_1} \cap T_{\B_2}$. 
Then its projection in $\N_n \times \N_n$ satisfies 
$\pi(p) \in \sigma_{\B_1} \cap \sigma_{\B_2} = \cone\{\e_{S|T} \, : \, S|T \text{ is a bisubset of } \B_1 \text{ and } \B_2\}$ since $\Sigma_{n,n}$ is a triangulation. Then $p \in \conv(v_{S|T} \, : \, S|T \text{ is a bisubset of } \B_1 \text{ and } \B_2\}$ analogously to the argument above.
It follows that
\[
T_{\B_1} \cap T_{\B_2} = \conv(v_{S|T} \, : \, S|T \text{ is a bisubset of } \B_1 \text{ and } \B_2\},
\]
which is a common face of $T_{\B_1}$ and $T_{\B_2}$.

\bigskip

2. Now we prove that $\T$ is unimodular. Consider a simplex $T_\B$ corresponding to a bipermutation $\B$. We know \cite{ArdilaDenhamHuh1} that the cone $\sigma_{\B}$ is unimodular, so $\{\e_{S|T} \, : \, S|T \subseteq \B\}$ is a basis for the lattice $(\Z^n/\Z) \oplus (\Z^n/\Z)$. It follows that $\{\e_S + \f_T : \, S|T \subseteq \B\} \cup \{\e_\emptyset + \f_E, \e_E + \f_\emptyset\}$ is a basis for $\Z^n \oplus \Z^n$. Since $\pi$ is a unimodular bijective transformation between $\Z^{3 \times n} \cap P$ and $\Z^n \oplus \Z^n$, it follows that the edges coming out of the vertex $v_{E,E}$ of $T_{\B}$ are a basis for $\Z^{3 \times n} \cap P$, as desired.

\medskip

3. It remains to observe that the map $T_{\B} \mapsto \sigma_{\B}$ is a bijection between the facets of the triangulation $\T$ and the facets of the fan $\Sigma_{n,n}$, and their patterns of intersection are identical thanks to the observation 1b. above. Combinatorially, the only difference between them is that triangulation $\T$ has the three cone points $v_{\emptyset,E}, v_{E,\emptyset}, v_{E,E}$.
\end{proof}

\section{A formula for the biEulerian polynomial and real-rootedness.}

Recall that the \emph{Ehrhart polynomial} $\textsf{ehr}_P(k)$ of a $d$-dimensional lattice polytope $P$ in $\R^n$ is defined by
\[
\textsf{ehr}_P(k) = |kP \cap \Z^n|
\]
and the \emph{Ehrhart $h^*$-polynomial} $h^*_P(x)$ is defined by
\[
\sum_{k \geq 0} \textsf{ehr}_P(k) x^k = \frac{h^*_P(x)}{(1-x)^{d+1}}.
\]
If $P$ has a unimodular triangulation, then its $h^*$-polynomial can be computed combinatorially as follows.

\begin{theorem} \label{th:h=h*} \cite[Theorem 10.3] {BeckRobins}
If $\T$ is a unimodular triangulation of a lattice polytope $P$, then
\[
h^*_P(x) = h_{\T}(x).
\]
\end{theorem}

Applying this to the bipermutahedral triangulation of the product of $n$ triangles constructed in Theorem \ref{th:triangulation}, we get the following result.

\begin{theorem} \label{th:biEulerianseries}
The biEulerian polynomial $B_n(x)$ is given by
\[
\frac{B_n(x)}{(1-x)^{2n+1}} = \sum_{k \geq 0} {k+2 \choose 2}^n x^k
\]
\end{theorem}

\begin{proof}
Consider the bipermutahedral triangulation of the product $\Delta^n$ of $n$ triangles of Theorem \ref{th:triangulation}, which is unimodular. The Ehrhart polynomial of $\Delta^n$ is
\[
\textsf{ehr}_{\Delta^n}(k) = 
(\textsf{ehr}_{\Delta}(k))^n
= {k+2 \choose 2}^n,
\]
so the series in the right hand side of the statement is the Ehrhart series of $\Delta^n$, and the numerator of the left hand side is its $h^*$-polynomial, which equals the $h$-polynomial of the unimodular triangulation $\T$ by Theorem \ref{th:h=h*}.

Now, the triangulation $\T$ is combinatorially isomorphic to a triple cone over the bipermutahedral fan. Furthermore, since the $f$-polynomials of a simplicial complex $\Delta$ and its cone $c\Delta$ are related by $f_{c\Delta}(x) = (x+1)f_{\Delta}(x)$, their $h$-polynomials are identical. Therefore $h_{\T}(x) = h_{\Sigma_{n,n}}(x) = B_n(x)$, as desired.
\end{proof}

\begin{theorem} \label{th:realrooted}
The biEulerian polynomial $B_n(x)$ is 
real-rooted. Consequently the $h$-vector of the bipermutahedral fan is log-concave and unimodal.
\end{theorem}

\begin{proof}
Wagner \cite{Wagner} defined a (non-linear) operator $\mathscr{W}: \R[x] \rightarrow \R[x]$ by the rule $f \longmapsto \mathscr{W}f$, where
\[
 \frac{\mathscr{W}f(z)}{(1-z)^{\deg f + 1}} = \sum_{k \geq 0} f(k)z^k 
\]
He also showed \cite[Theorem 0.2]{Wagner} that if $f$ and $g$ are polynomials such that all the roots of $\mathscr{W}f$ and $\mathscr{W}g$ are real and non-positive, then all the roots of $\mathscr{W}(fg)$ are real and non-positive as well. Since
\[
\mathscr{W}{x+2 \choose 2} = h_1(x) = 1
\]
satisfies this condition vacuously, then
\[
\mathscr{W}\left({x+2 \choose 2}^n\right) = h_n(x)
\]
satisfies it as well. Finally we remark that the coefficients of a real-rooted polynomial are log-concave; and if they are positive, then they are unimodal. \cite[Lemma 1.1]{Branden}
\end{proof}

\section{Symmetries}

The simple inequality description of the bipermutahedron given in  Proposition \ref{prop:inequalities} makes us suspect that $\Pi_{n,n}$ is one of the most elegant polytopes whose normal fan is the bipermutahedral fan $\Sigma_{n,n}$. In this section we make this suspicion precise, proving in Proposition \ref{prop:symmetry} that $\Pi_{n,n}$ is maximally symmetric. Let $S_n$ be the symmetric group on $[n]$.

\begin{proposition}\label{prop:automorphisms}
The automorphism group of the permutahedral fan $\Sigma_{n,n}$ is $S_n \times \Z_2 \times \R_{>0}$.
\end{proposition}

\begin{proof}
We will prove that the automorphisms of $\Sigma_{n,n}$ are:

$\bullet$ the simultaneous action $\sigma(z_1, \ldots, z_n, w_1, \ldots, w_n) = (z_{\sigma(1)}, \ldots z_{\sigma(n)}, w_{\sigma(1)}, \ldots, w_{\sigma(n)})$ of a permutation $\sigma \in S_n$ on the two factors of $\N_n \times \N_n$.

$\bullet$ the \emph{swap} $f(z,w) = (w,z)$, 

$\bullet$ the positive dilations $f(z,w) = (rz, rw)$ for $r>0$,

\noindent 
and their compositions. This will prove the desired result. 
The case $n=2$ can be checked by inspection of Figure \ref{fig:biperm}, so we will assume hereon that $n \geq 3$.

\medskip

Consider any automorphism $g \in GL(\N_n \times \N_n)$ of the bipermutahedral fan. Rescaling $g$ by a positive constant does not affect the invariance of $\Sigma_{n,n}$, so we may assume that $g$ is unimodular; that is, $\det g = \pm 1$. 
Consider the hyperplanes spanned by the walls of the permutahedral fan; the automorphism $g$ must send each one of these spanned hyperplanes to another spanned hyperplane. 
The codimension 1 faces of $\Sigma_{n,n}$ lie on hyperplanes of the form
\[
1. \,\, z_i+w_i = z_j+w_j, \qquad 
2. \,\, z_i = z_j, \qquad 
3. \,\, w_i = w_j, \qquad 
4. \,\, z_i-z_k = w_k-w_j
\]
for $i,j,k \in [n]$. Let us count the number of codimension 1 faces of $\Sigma_{n,n}$ on each of these hyperplanes.

\smallskip

\noindent
1. The codimension 1 faces on hyperplane $z_i+w_i = z_j+w_j$ are those indexed by bipermutations of the form $b_1| \ldots | \mathbf{i} | \ldots | \mathbf{j} | \ldots | b_{2n-2}$. They are in bijective correspondence with the permutations of the multiset $(E-i-j) \cup (E-i-j) \cup \mathbf{i} \cup \mathbf{j}$, so the number of them is $(2n-2)!/2^{n-2}$.

\smallskip

\noindent
2. The codimension 1 faces on hyperplane $z_i = z_j$ are indexed by bipermutations of three types:

\smallskip

a) Bipermutations of the form $b_1| \ldots | ij | \ldots  | b_{2n-2}$ where the first $i$ and the first $j$ are in the same block. To specify such a bipermutation, we need to choose the element $\mathbf{k}$ that appears only once, and then choose a permutation of $(E-i-j-k) \cup (E-i-j-k) \cup \mathbf{k} \cup ij \cup \overline{i} \cup \overline{j}$ where $ij$ precedes both $\overline{i}$ and $\overline{j}$. There are $(n-2) \cdot [(2n-2)!/2^{n-3}]/3$ such permutations.

\smallskip

b. Bipermutations of the form $b_1| \ldots | i\mathbf{j} | \ldots  | b_{2n-2}$ where the first $i$ and $\mathbf{j}$ are in the same block. To specify such a bipermutation, we need  choose a permutation of $(E-i-j) \cup (E-i-j) \cup i\mathbf{j} \cup \overline{i}$ where $i\mathbf{j}$ precedes $\overline{i}$. There are $[(2n-2)!/2^{n-2}]/2$ such permutations.

\smallskip

c. Bipermutations of the form $b_1| \ldots | \mathbf{i}j | \ldots  | b_{2n-2}$ where $\mathbf{i}$ and the first $j$ are in the same block. Similarly, there are $[(2n-2)!/2^{n-2}]/2$ such permutations.

\smallskip

\noindent
Therefore the total number of codimension 1 faces on this hyperplane is $(2n-1)!/(3 \cdot 2^{n-2})$.

\smallskip

\noindent
3. The number of codimension 1 faces on hyperplane $w_i = w_j$ is also $(2n-1)!/(3 \cdot 2^{n-2})$.

\smallskip

\noindent
4. The faces on hyperplane $z_i - z_k = w_k- w_j$ are those indexed by bipermutations of the form $b_1| \ldots | i\overline{j} | \ldots | \mathbf{k} | \ldots | b_{2n-2}$ where the first $i$ and the second $j$ are in the same block. The number of them is the number of permutations of the multiset $(E-i-j-k) \cup (E-i-j-k) \cup \mathbf{k} \cup i\overline{j} \cup \overline{i} \cup j$ where $i\overline{j}$ comes after $j$ but before $\overline{i}$. There are $[(2n-2)!/2^{n-3}]/6$ such permutations.

\smallskip

The numbers $(2n-2)!/2^{n-2}$, \, $(2n-1)!/(3 \cdot 2^{n-2})$\, and $[(2n-2)!/2^{n-3}]/6$ are different for $n \geq 3$. Therefore the automorphism $g$ must map hyperplanes of type 1 to hyperplanes of type 1, it must map 
hyperplanes of type 4 to hyperplanes of type 4,
and
it must map hyperplanes of types 2 and 3 to hyperplanes of types 2 and 3.

\bigskip

Let $x_i=z_i+w_i$ and consider the braid arrangement given by hyperplanes $x_i=x_j$ for $i \neq j$. These are the hyperplanes of type 1 above, so the automorphism $g$ must leave this arrangement invariant.
The hyperplanes $x_1=x_2, \, x_2=x_3, \ldots,x_{n-1}=x_n$ cut out precisely two chambers of the braid arrangement, namely $x_1 > x_2 > \cdots > x_n$ and $x_1 < x_2 < \cdots < x_n$. Therefore the images of these $n-1$ hyperplanes under the automorphism $g$ must also cut out two chambers of the braid arrangement; thus they must be of the form $x_{\sigma(1)} = x_{\sigma(2)}, \, x_{\sigma(2)} = x_{\sigma(3)}, \ldots, x_{\sigma(n-1)} = x_{\sigma(n)}$, respectively, for some permutation $\sigma \in S_n$. 


The action of $g$ on $\N_n \times \N_n$ is equivalent to the action of $g$ on the dual space $\M_n \times \M_n$ where if $m \in \M_n \times \M_n$ then $g \cdot m$ is given by $g \cdot m (n) = m(g^{-1} \cdot n)$ for $n \in \N_n \times \N_n$.
Consider the normal vectors $\pm(\d^i-\d^j)$ to hyperplane $x_i=x_j$, where $\d^i := \e^i + \f^i \in \M_n \times \M_n$.
To preserve lengths, $g$ must send $\d^i - \d^{i+1}$ to one of the vectors $\pm(\d^{\sigma(i)} - \d^{\sigma(i+1)})$. To preserve the angles between these vectors -- computed through their dot products -- we must have one of two cases:
\[
1) \,\,\,
g \cdot (\d^1 - \d^2) = \d^{\sigma(1)} - \d^{\sigma(2)}, \,\,\, 
g \cdot (\d^2 - \d^3) = \d^{\sigma(2)} - \d^{\sigma(3)}, \,\,\, 
\ldots,
g \cdot (\d^{n-1} - \d^n) = \d^{\sigma(n-1)} - \d^{\sigma(n)}
\]
or 
\[
2) \,\,\,
g \cdot (\d^1 - \d^2) = \d^{\sigma(2)} - \d^{\sigma(1)}, \,\,\, 
g \cdot (\d^2 - \d^3) = \d^{\sigma(3)} - \d^{\sigma(2)}, \,\,\, 
\ldots,
g \cdot (\d^{n-1} - \d^n) = \d^{\sigma(n)} - \d^{\sigma(n-1)}
\]
Let us assume that we are in the first case. Since $g$ maps hyperplanes of types $2$ and $3$ to each other, it maps its normal vectors $\pm(\e^i-\e^j)$ and $\pm(\f^i-\f^j)$ to each other. Therefore, 
$g \cdot (\e^i+\f^i - \e^{i+1}-\f^{i+1}) = \e^{\sigma(i)} + \f^{\sigma(i)} - \e^{\sigma(i+1)} - \f^{\sigma(i+1)}$ can only hold if
\[
1a) \,\,\, g \cdot (\e^i-\e^{i+1}) = \e^{\sigma(i)} - \e^{\sigma(i+1)}, \,\,\,
g \cdot (\f^i-\f^{i+1}) = \f^{\sigma(i)} - \f^{\sigma(i+1)} \qquad \text{ for all } i,
\]
or 
\[
1b) \,\,\, g \cdot (\e^i-\e^{i+1}) = \f^{\sigma(i)} - \f^{\sigma(i+1)}, \,\,\,
g \cdot (\f^i-\f^{i+1}) = \e^{\sigma(i)} - \e^{\sigma(i+1)} \qquad
\text{ for all } i.
\]

In case 1a), $g$ acts like the same permutation $\sigma$ on the first and second factors of $\N_n \times \N_n$.
In case 1b), $g$ acts as above, followed by the swap map $s(z,w) = (w,z)$. These are indeed automorphisms of $\Sigma_{n,n}$.

In cases 2), $g$ acts as in 1a) or 1b), followed by the transformation $r(z,w) = -(z,w)$. However, $r$ is \textbf{not} an automorphism of the bipermutahedral fan. To see this, choose a bisubset $S|T$ with $S, T \neq [n]$ and $S \cap T \neq \emptyset$; this is possible for $n \geq 3$. Then $r$ maps the ray $\e_S + \f_T$ of $\Sigma_{n,n}$ to $\e_{E-S} + \f_{E-T}$, which is not a ray of $\Sigma_{n,n}$. Therefore cases 2a) and 2b) do not lead to automorphisms of $\Sigma_{n,n}$. 
The desired result follows. \end{proof}

\begin{proposition} \label{prop:symmetry}
The automorphism group of the bipermutahedron $\Pi_{n,n}$ is $S_n \times \Z_2$. This is the largest automorphism group among all polytopes whose normal fan is the bipermutahedral fan.
\end{proposition}

\begin{proof}
It is clear from the inequality description of Proposition \ref{prop:inequalities} that the bipermutahedron $\Pi_{n,n}$ is fixed by the simultaneous action of a permutation on both factors of $\M_n \times \M_n$, and by the swap map $s(z,w) = (w,z)$. On the other hand, dilations by positive constants other than $1$ cannot preserve a polytope. In view of Proposition \ref{prop:automorphisms}, the result follows.
\end{proof}

\section{The ample and nef cones: deformations of the bipermutahedron}\label{sec:amplecone}

Our next goal is to describe \textbf{all} the \emph{deformations of the bipermutahedron}. These are the polytopes whose normal fan is a coarsening of the bipermutahedral fan $\Sigma_{n,n}$. This set of polytopes is known as the \emph{nef cone} of the fan; it is a cone because it is closed under positive dilation and under Minkowski sums. Its interior, known as the \emph{ample cone}, consists of the polytopes whose normal fan equals $\Sigma_{n,n}$. A priori it is not clear that the ample cone is non-empty, but it is shown in \cite{ArdilaDenhamHuh1} that the bipermutahedron lives in the ample cone, while the harmonic polytope lives in the nef cone.

We will show that the nef cone of the bipermutahedral fan is cut out by two kinds of inequalities:

\medskip
\noindent
\textsf{A) Supermodular inequalities:}
Let $\B = b_1| \cdots |b_{h-1}| ij |  b_{h+1}| \cdots | b_{2n-2}$ 
be a bisequence of length $2n-2$ consisting of $2n-1$ singletons and one pair, and let $S=\{b_1, \ldots, b_{h-1}\}$ and $T=\{b_{h+1}, \ldots, b_{2n-1}\}$.  The corresponding \emph{supermodular inequality} is
\[
I_{\B}(h) := \big(h(S|ijT) + h(Sij|T)\big) - \big(h(Si|Tj) + h(Sj|Ti)\big) \geq 0
\]

\medskip
\noindent
\textsf{B) Up-down inequalities}: Let $\B = b_1 | \ldots | {\bf i} | \ldots | {\bf j} | \ldots |b_{2n-2}$ be a bisequence of length $2n-2$ consisting of $2n-2$ singletons, with non-repeated elements $i$ and $j$. Let $\overline{\B} = b_1 | \ldots | i | \overline{i} | \ldots | j | \overline{j} | \ldots |b_{2n-2}$, where as usual we write  $h$ and $\overline{h}$ for the first and second occurrences of each number $h$ in $\overline{\B}$.

Consider the $m$th bisubset $S|T$ of $\overline{\B}$ where $S=\{b_1, \ldots, b_m\}$ and $T=\{b_{m+1}, \ldots, b_{2n-2}\}$. 
If $b_m$ is unbarred and $b_{m+1}$ is barred in $\overline{\B}$, we say $S|T$ is an \emph{up bisubset} of $\overline{\B}$, and write $S|T  \curlyeqprec \overline{\B}$.
If $b_m$ is barred and $b_{m+1}$ is unbarred in $\overline{\B}$, we say $S|T$ is a \emph{down bisubset} of $\overline{\B}$, and write $S|T  \preccurlyeq \overline{\B}$.
%
Then the \emph{up-down inequality} associated to $\B$ is:
\[
I_{\B}(h) := \Big(\sum_{S|T \preccurlyeq \overline{\B}} h(S|T) \Big) - \Big(\sum_{S|T \curlyeqprec \overline{\B}} h(S|T) \Big) \geq 0.
\]

\begin{example} The following are examples of supermodular and up-down inequalities for $n=7$.

\smallskip
\noindent
A) If $\B = 7|2|{\bf 3}|4|2|14|5|1|5|6|6|7$ then the corresponding supermodular inequality is
\[
h(2347|14567) + h(12347|1567) \geq
h(12347|14567) + h(2347|1567) 
\]

\smallskip
\noindent
B) If $\B = 7|2|{\bf 3}|4|2|4|5|{\bf 1}|5|6|6|7$ then $\overline{\B} = 7|2|3\cyan{\mathbf{|}}\3\magenta{\mathbf{|}}4\cyan{\mathbf{|}}\2|\4\magenta{\mathbf{|}}5|1\cyan{\mathbf{|}}\1|\5\magenta{\mathbf{|}}6\cyan{\mathbf{|}}\6|\7$ 
where we mark the up switches (resp. down switches) from unbarred to barred (resp. from barred to unbarred) elements in cyan (resp. magenta). Those switches determine the corresponding up-down inequality:
\begin{eqnarray*}
&& h(237|124567) + h(2347|1567) + h(123457|67) \\
&\geq& {h(237|1234567)} + h(2347|124567) + {h(123457|1567)} + h(1234567|67). 
\end{eqnarray*}
\end{example}

\begin{proposition}\label{prop:deformations}
The polytopes in $\M_n \times \M_n$ whose normal fan is the bipermutahedral fan $\Sigma_{n,n}$ are those of the form
\begin{eqnarray*}
\sum_{e \in [n]} x_e &=& 0, \\
\sum_{e \in [n]} y_e &=& 0 , \\
\sum_{s \in S} x_s + \sum_{t \in T} y_t &\geq& h(S|T) \qquad \text{for each bisubset $S|T$ of $[n]$}
\end{eqnarray*}
where the function $h$ strictly satisfies the supermodular and up-down inequalities.
\end{proposition}

\begin{proof}
There is a general \emph{Wall-Crossing Criterion} \cite[Theorems 6.1.5--6.1.7]{Cox-Little-Schenck} that describes the nef cone of a convex fan $\Sigma$. Let us state it in the case of complete simplicial fans $\Sigma$ in a vector space $\N$ of dimension $d$.
Let $\RR(\Sigma)$ be a set of vectors that generate the rays of $\Sigma$, with one vector for each ray.
 Let $\tau$ be a codimension $1$ face of $\Sigma$, or \emph{wall}, that separates two full-dimensional chambers $\sigma$ and $\sigma'$ of $\Sigma$.  
Consider the rays $\r_1, \ldots, \r_{d-1}, \r, \r' \in \RR(\Sigma)$ such that 
\[
\tau=\cone(\r_1, \ldots, \r_{d-1}), \quad 
\sigma=\cone(\r_1, \ldots, \r_{d-1}, \r), \quad 
\sigma'=\cone(\r_1, \ldots, \r_{d-1}, \r')
\]
Up to scaling, there is a unique linear dependence of the form
\begin{equation}\label{eq:wallcross}
\displaystyle c\cdot {\r} + c'\cdot {\r'} =  \sum_{i=1}^{d-1}c_i\cdot {\r_{i}} 
\end{equation}
with $c, c' >0$. To the wall $\tau$ we associate the \emph{wall-crossing inequality}
\begin{equation}\label{ineq:wallcross}
\displaystyle  
I_{\Sigma,\tau}(h) :=  
c\cdot h({\r}) + c'\cdot h({\r'}) - \sum_{i=1}^{d-1}c_i\cdot h({\r_{i}}) \geq 0.
\end{equation}
Then the nef cone of $\Sigma$ consists of the polytopes in the dual space $\M = \N^*$ of the form
\[
P(h) = \{x \in \M \, : \, \r(x) \leq h(\r) \text{ for all rays } \r \in \RR(\Sigma)\}
\]
for the functions $h: \RR(\Sigma) \rightarrow \R$ that satisfy the wall-crossing inequalities \eqref{ineq:wallcross}.

\medskip

Let us apply the Wall-Crossing Criterion (reversing all inequalities) to the bipermutahedral fan $\Sigma_{n,n}$. It contains two kinds of walls, corresponding to the two possible kinds of bisequences of length $2n-2$.

\medskip
\noindent
A) The wall $\tau$ given by bisequence $b_1| \cdots |b_{h-1}| ij |  b_{h+1}| \cdots | b_{2n-1}$, which separates the chambers $\sigma$ and $\sigma'$ given by bipermutations $\B = b_1| \cdots | i | j | \cdots | b_{2n-1}$ and $\B' =  b_1| \cdots | j | i | \cdots | b_{2n-1}$ for $i \neq j$.

\medskip

The rays $\r \in \RR(\sigma) - \RR(\tau)$ and $\r' \in \RR(\sigma')-\RR(\tau)$ are  $\r = \e_{Si|Tj}$ and $\r'=\e_{Sj|Ti}$ for the sets $S=\{b_1, \ldots, b_{h-1}\}$ and $T=\{b_{h+1}, \ldots, b_{2n-1}\}$. The equation \eqref{eq:wallcross} is $\e_{Si|Tj} + \e_{Sj|Ti} = \e_{S|ijT} + \e_{Sij|T}$ in this case, so the wall-crossing inequality is
\[
h(Si|Tj) + h(Sj|Ti) \leq h(S|ijT) + h(Sij|T).
\]

\medskip
\noindent
B) The wall $\tau$ given by bisequence $\B_\tau = b_1| \cdots | {\bf i} | \cdots | {\bf j} | \cdots | b_{2n-1}$, which separates the chambers $\sigma$ and $\sigma'$  given by bipermutations $\B = b_1| \cdots | i | i | \cdots | {\bf j} | \cdots | b_{2n-1}$ and $\B' = b_1| \cdots | {\bf i} | \cdots | j | j | \cdots | b_{2n-1}$.

\medskip

The wall-crossing inequality can be nicely understood in terms of a bipartite graph $\Gamma(\B_\tau)$, defined as follows; see Figure \ref{fig:tree} for an example.

\smallskip
\noindent 
\emph{Vertices:} The top vertices are the $n$ distinct sets of the form $\{b_1, \ldots, b_i\}$ for $1 \leq i \leq 2n-2$ and the bottom vertices are the $n$ distinct sets of the form $\{b_i, \ldots, b_{2n-2}\}$ for $1 \leq i \leq 2n-2$. 

\smallskip
\noindent
\emph{Edges:} Let $\overline{\B_\tau} = b_1 | \ldots | i |\overline{i} | \ldots | j | \overline{j} | \ldots |b_{2n-2}$. Each of the $2n-1$ bisubsets $S|T$ of $\overline{\B_\tau}$ induces an edge connecting the top vertex $S$ to the bottom vertex $T$ in $\Gamma(\B_\tau)$. The two special edges $e$ and $e'$ corresponding to the two splits at $i|\overline{i}$ and at $j|\overline{j}$ are drawn with thick lines.

\smallskip
\noindent 
\emph{The spine:} Since the splits of $\overline{\B_\tau}$ are linearly ordered, the edges of $\Gamma(\B_\tau)$ are linearly ordered from left-to-right, and cannot cross. Therefore $\Gamma(\B_\tau)$ has no cycles; and since it has $2n$ vertices and $2n-1$ edges, it is a tree. Thus there is a unique path that connects the bottom left vertex $[n]$ to the top right vertex $[n]$: its edges correspond to the places where the permutation $\pi(\B_\tau)$ switches between barred and unbarred elements. Therefore this path contains the two special edges $e$ and $e'$.
We call this the \emph{spine} of $\Gamma$, and mark it with thick lines, alternating in color between cyan and magenta; the special edges $e$ and $e'$ are both cyan.

\begin{figure}[h]
\[
\Gamma(7|2|{\bf 3}|4|2|4|5|{\bf 1}|5|6|6|7) = 
\begin{tikzpicture}[scale=1.25,baseline=(current bounding box.center),
plain/.style={circle,draw,inner sep=1.5pt,fill=white},
root/.style={circle,draw,inner sep=1.5pt,fill=black}]
\node (1b) at (1.1,0) [root, label=below: $1234567$] {};
\node (2b) at (2.2,0) [root, label=below: $124567$] {};
\node (3b) at (3.3,0) [root, label=below: $14567$] {};
\node (4b) at (4.4,0) [root, label=below: $1567$] {};
\node (5b) at (5.5,0) [root, label=below: $567$] {};
\node (6b) at (6.6,0) [root, label=below: $67$] {};
\node (7b) at (7.7,0) [root, label=below: $7$] {};
\node (2a) at (1.1,1.5) [root, label=above: $7$] {};
\node (3a) at (2.2,1.5) [root, label=above: $27$] {};
\node (4a) at (3.3,1.5) [root, label=above: $237$] {};
\node (5a) at (4.4,1.5) [root, label=above: $2347$] {};
\node (6a) at (5.5,1.5) [root, label=above: $23457$] {};
\node (7a) at (6.6,1.5) [root, label=above: $123457$] {};
\node (8a) at (7.7,1.5) [root, label=above: $1234567$] {};
\draw (2a) -- (1b);
\draw (3a) -- (1b);
\draw [line width=1mm, cyan] (4a) -- (1b);
\draw  [line width=0.5mm, magenta](4a) -- (2b);
\draw  [line width=0.5mm, cyan](5a) -- (2b);
\draw (5a) -- (3b);
\draw  [line width=0.5mm, magenta](5a) -- (4b);
\draw (6a) -- (4b);
\draw [line width=1mm, cyan] (7a) -- (4b);
\draw (7a) -- (5b);
\draw  [line width=0.5mm, magenta](7a) -- (6b);
\draw  [line width=0.5mm, cyan](8a) -- (6b);
\draw (8a) -- (7b);
\end{tikzpicture}
\]
\caption{The bipartite graph for 
$\B_\tau=7|2|{\bf 3}|4|2|4|5|{\bf 1}|5|6|6|7$ and  
$\overline{\B_\tau} = 7|2|3\cyan{\mathbf{|}}\3\magenta{\mathbf{|}}4\cyan{\mathbf{|}}\2|\4\magenta{\mathbf{|}}5|1\cyan{\mathbf{|}}\1|\5\magenta{\mathbf{|}}6\cyan{\mathbf{|}}\6|\7$.
\label{fig:tree}}
\end{figure}

The rays $\r_1, \ldots, \r_{2n-2}$ of $\tau$ correspond to the ordinary edges of $\Gamma$ and the rays $\r = \RR(\sigma) - \RR(\tau)$ and 
$\r' = \RR(\sigma') - \RR(\tau)$ correspond to the two special magenta edges of $\Gamma$. Notice that the alternating sum of the rays corresponding to the spine of $\Gamma$ equals $\e_E + \f_E = 0$ in $\N_n \times \N_n$. In the example above this equality reads
\[
\e_{237|1234567} - \e_{237|124567} 
+\e_{2347|1234567} - \e_{237|1567} 
+\e_{123457|1567} - \e_{123457|67} 
+ \e_{1234567|67} = 0 \quad \text{in } \N_7 \times \N_7.
\]
This must be the unique wall-crossing dependence \eqref{eq:wallcross}, so
the wall-crossing inequality for the wall $\tau$ is precisely the up-down inequality for the bisequence $B_\tau$.
\end{proof}

It is not at all clear from Proposition \ref{prop:deformations} whether a bipermutahedron exists; that is, whether the ample cone (the interior of the nef cone) of the bipermutahedral fan is non-empty. We do know that it is non-empty, because it contains the support function $\Pi(S|T) = - \big(|S| + |S-T|\big) \cdot \big(|T| + |T-S| \big)$ of the bipermutahedron. However, even with such a simple, explicit description, it is not so easy to see why this function satisfies the wall-crossing inequalities!

\section{The Minkowski quotient of the bipermtuahedron and the harmonic polytope is 2}

It is instructive to verify directly that the support function of the bipermutahedron and the harmonic polytope satisfy the wall-crossing inequalities, and we do so in this section. This computation has a stronger,  unexpected consequence: it implies that in any dimension, the \emph{Minkowski quotient} of the bipermutahedron $\Pi_{n,n}$ and the harmonic polytope $H_{n,n}$ equals $2$.

\begin{proposition} \label{prop:bipermineqs}
The support function of the bipermutahedron $\Pi_{n,n}$
\[
\Pi(S|T) = - \big(|S| + |S-T|\big) \cdot \big(|T| + |T-S| \big) \qquad \text{for each bisubset $S|T$ of $[n]$}.
\]
satisfies the strict wall-crossing inequalities of the bipermutahedral fan.
\end{proposition}

\begin{proof}
We already know this statement must be true because the normal fan of the bipermutahedron $\Pi_{n,n}$ is the bipermutahedral fan, so its support function $\Pi$ must be in the ample cone of $\Sigma_{n,n}$. However, we wish to give a direct proof that will allow us to derive a stronger result. 

\medskip

\noindent \textsf{A) Supermodular inequalities:}
Let $\B = b_1| \cdots |b_{h-1}| ij |  b_{h+1}| \cdots | b_{2n-2}$ 
be a bisequence of length $2n-2$ consisting of $2n-1$ singletons and one pair. Let $S=\{b_1, \ldots, b_{h-1}\}$ and $T=\{b_{h+1}, \ldots, b_{2n-1}\}$.  
Let $s=|S|, t=|T|,$ and $u = |S \cap T|$. 
Since $i$ and $j$ appear in the $h$th part of $\B$, each one appears at most once in the remaining entries of $\B$, and at least one of them must appear. Thus we have three cases, where the computations are straightforward:

\smallskip
\noindent
(i) \emph{$i$ and $j$ appear on the same side of $ij$ in $\B$; say $i,j \in S$.}
In this case we have
\begin{eqnarray*} 
|S| = s, \,\,\, |Tij| = t+2, \,\,\, |S \cap (Tij)| = u+2, \qquad &&
|Sij| = s, \,\,\, |T| = t, \,\,\, |(Sij) \cap T)| = u, \\
|Si| = s, \,\,\, |Tj| = t+1, \,\,\, |(Si) \cap (Tj)| = u+1, \qquad &&
|Sj| = s, \,\,\, |Ti| = t+1, \,\,\, |(Sj) \cap (Ti))| = u+1,
\end{eqnarray*}
from which the corresponding supermodular inequality follows readily:
\begin{eqnarray*}
I_{\B}(\Pi) &:=&  \big(\Pi(S|ijT) + \Pi(Sij|T)\big) - \big(\Pi(Si|Tj) + \Pi(Sj|Ti)\big) \\
&=& 
-[s+(s-u-2)] \cdot [(t+2)+(t-u)] -[s+(s-u)] \cdot[t+(t-u)] \\
& & + \, [s+(s-u-1)] \cdot [(t+1)+(t-u)] + [s+(s-u-1)] \cdot [(t+1)+(t-u)] \\
&=& 2 > 0.
\end{eqnarray*}

\smallskip
\noindent
(ii) \emph{$i$ and $j$ appear on different sides of $ij$ in $\B$; say $i \in S$ and $j \in T$.}
Similarly,
\begin{eqnarray*}
I_{\B}(\Pi) &=&  
-[s+(s-u-1)] \cdot [(t+1)+(t-u)] -[(s+1)+(s-u)] \cdot[t+(t-u-1)] \\
& & + \, [s+(s-u)] \cdot [t+(t-u)] + [(s+1)+(s-u-1)] \cdot [(t+1)+(t-u-1)] \\
&=& 2 > 0.
\end{eqnarray*}

\smallskip
\noindent
(iii) \emph{Only one of $i$ and $j$ appears again in $\B$; say $i \in S$ and $j \notin S,T$.}
Similarly,
\begin{eqnarray*}
I_{\B}(\Pi) &=&  
-[s+(s-u-1)] \cdot [(t+2)+(t-u+1)] -[(s+1)+(s-u+1)] \cdot[t+(t-u)] \\
& & + \, [s+(s-u)] \cdot [(t+1)+(t-u+1)] + [(s+1)+(s-u)] \cdot [(t+1)+(t-u)] \\
&=& 4 > 0.
\end{eqnarray*}

\medskip
\noindent
\textsf{B) Up-down inequalities}: Let $\B = b_1 | \ldots | {\bf i} | \ldots | {\bf j} | \ldots |b_{2n-2}$ be a bisequence of length $2n-2$ consisting of $2n-2$ singletons, with non-repeated elements $i$ and $j$. Let $\overline{\B} = b_1 | \ldots | i | \overline{i} | \ldots | j | \overline{j} | \ldots |b_{2n-2}$.

Proving that $I_\B(\Pi) \geq 0$ is more interesting in this case; we do it by interpreting this quantity as an area. 
Let us draw a $2n \times 2n$ square board whose rows and columns are indexed by the entries of $\overline{\B}$, and draw a vertical and horizontal lines 
where there are switches between barred and unbarred labels. There is one intersection point along the main diagonal for each switch, and thus for each term of the wall-crossing inequality $I_\B(\Pi) \geq 0$. 
Figure \ref{fig:table0} illustrates this construction for $\B = 7|2|{\bf 3}|4|2|4|5|{\bf 1}|5|6|6|7$. 

Let us analyze one of the terms $\Pi(S|T) = - \big(|S| + |S-T|\big) \cdot \big(|T| + |T-S| \big)$ of the inequality, corresponding to a switch between a barred and an unbarred element in $\overline{\B}$, and to an intersection points $p$ along the diagonal. Since the bisubset at that switch equals $S|T$ and $i$ precedes $\overline{i}$ for all $i$, the rows above $p$ are indexed by $S \cup \overline{S-T}$ while the columns to the right of $p$ are indexed by $(T-S) \cup \overline{T}$. Therefore $-\Pi(S|T)$ is precisely the area of the rectangle going from $p$ to the top right corner of the square.
In the example of Figure \ref{fig:table0}, for the switch from $4$ to $\2$, we have
$-\Pi(2347 | 124567) = (|2347| + |\3|) \cdot (|\1\2\4\5\6\7| + |156|) = 5 \cdot 9 = 45$. 

\begin{figure}[h]
\begin{center}
\includegraphics[height=3in]{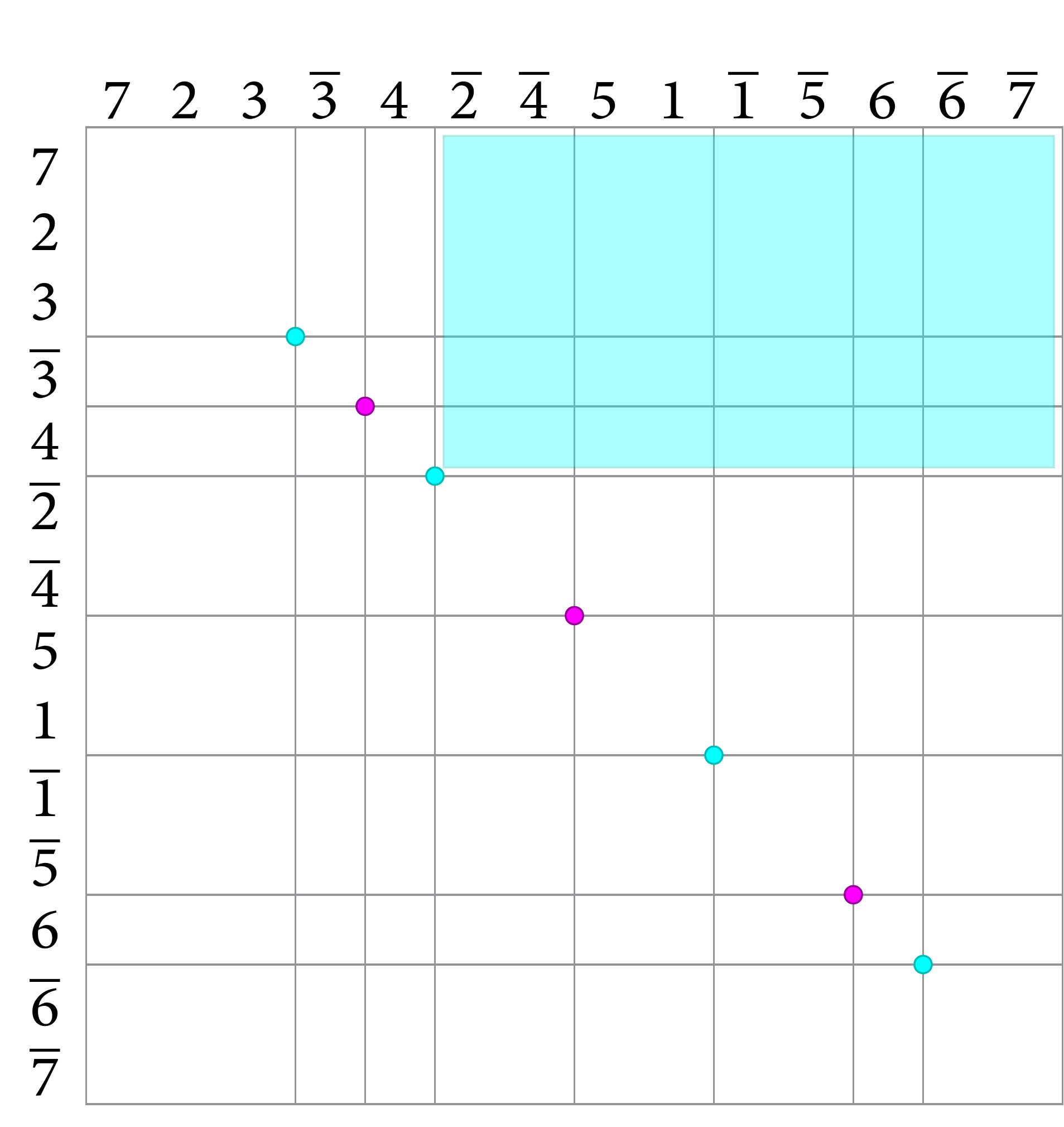}
\captionof{figure}{The table for $\B = 7|2|{\bf 3}|4|2|4|5|{\bf 1}|5|6|6|7$ and the area interpretation of $-\Pi(2347 | 124567)$.
\label{fig:table0}}
\end{center}
\end{figure}
Thus we may interpret the up-down inequality associated to $\B$
\[
I_{\B}(\Pi) := \Big(\sum_{S|T \preccurlyeq \overline{\B}} \Pi(S|T) \Big) - \Big(\sum_{S|T \curlyeqprec \overline{\B}} \Pi(S|T) \Big) \geq 0.
\]
as an alternating sum of areas that should be positive. This is best understood graphically, as shown in Figure \ref{fig:wall-crossing}. The figure verifies the up-down inequality for $\B = 7|2|{\bf 3}|4|2|4|5|{\bf 1}|5|6|6|7$, namely
\begin{eqnarray*}
&& \Pi(237|124567) + \Pi(2347|1567) + \Pi(123457|67) \\
&\geq& {\Pi(237|1234567)} + \Pi(2347|124567) + {\Pi(123457|1567)} + \Pi(1234567|67),
\end{eqnarray*}
but this graphical argument is entirely general. 

\begin{figure}[h]
\begin{center}
\includegraphics[width=6.5in]{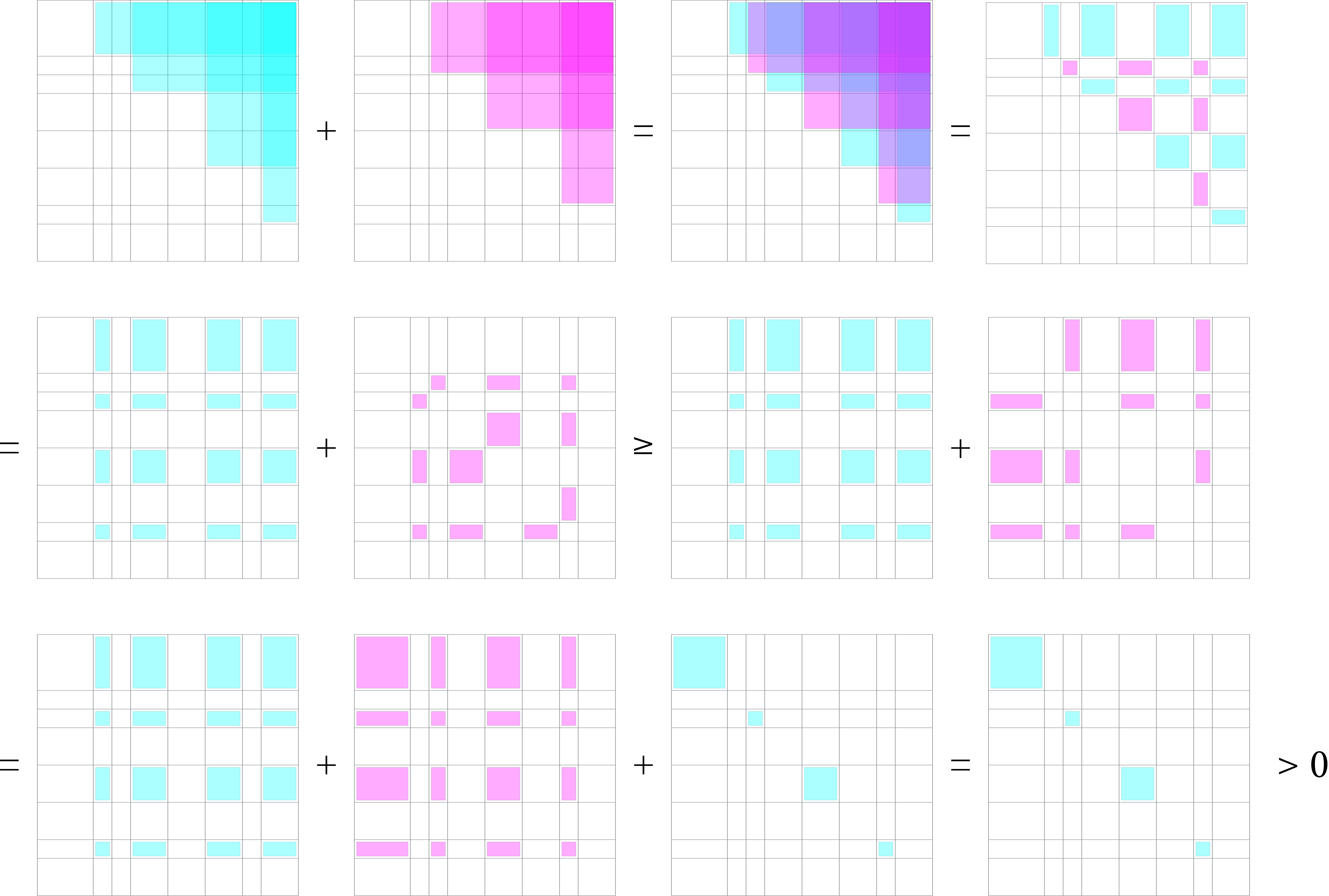}
\end{center}
\caption{The up-down inequality $I_\B(\Pi) \geq 0$ for $\B = 7|2|{\bf 3}|4|2|4|5|{\bf 1}|5|6|6|7$. Positive areas are shown in cyan and negative areas are shown in magenta.
\label{fig:wall-crossing}}
\end{figure}

The inequality in this graphical computation deserves an explanation. In the $(2i+1)$th column, we are sliding up $i$ (negative) magenta rectangles, replacing them with $i$ new magenta rectangles of larger total area; this is because every row index $\overline{j}$ that appeared among the first $i$ rectangles, the row index $j$ -- which precedes $\overline{j}$ in $\overline{\B}$ -- must also appear among the second $i$ rectangles. The same argument holds for the rows. 

The last equality also deserves an explanation. The $k \times k$ grid of cyan rectangles and the $k \times k$ grid of magenta rectangles both have area $n^2$, because their column labels and row labels are either $\{1, \ldots, n\}$ or $\{\1, \ldots, \n\}$.

Finally, let us remark that the last step actually shows the stronger inequality
\[
I_\B(\Pi) \geq n,
\]
since the smallest possible area of a set of squares whose side lengths are integers adding up to $n$ is $1^2+\cdots + 1^2 = n$.

\smallskip

Though it is perhaps less enlightening, we may rewrite this argument algebraically as follows. 
Let $a_1, a_2, \ldots, a_{2k-1}, a_{2k}$ be the lengths of the consecutive strings of barred and unbarred subsequences of $\overline{\B}$; these are the lengths of the segments along the edges of the square. For $\B = 7|2|{\bf 3}|4|2|4|5|{\bf 1}|5|6|6|7$  we have $(a_1, \ldots, a_8) = (3,1,1,2,2,2,1,2)$. Then

\begin{eqnarray*} 
I_\B(\Pi) &= &\sum_{i \text{ odd}} \Big(\sum_{j \leq i} a_j\Big) \Big(\sum_{j > i} a_j\Big) - 
\sum_{i \text{ even}} \Big(\sum_{j \leq i} a_j\Big) \Big(\sum_{j > i} a_j\Big) 
=  \sum_{2i+1 < 2j} a_{2i+1} a_{2j} - \sum_{2i < 2j+1} a_{2i}a_{2j+1}  \\
&=& \Big(\sum a_{2i+1}a_{2j}\Big) - 2 \Big(\sum_{2i < 2j+1} a_{2i}a_{2j+1}\Big) \,\, \geq \,\, 
\Big(\sum a_{2i+1} a_{2j}\Big) - 2\Big(\sum_{2i+1 < 2j+1} a_{2i+1}a_{2j+1}\Big) \\
&=& \Big(\sum a_{2i+1}\Big) \Big(\sum a_{2j}\Big) - \Big(\sum a_{2i+1} \Big)^2 + \Big(\sum a_{2i+1}^2\Big) \,\, = \,\,  \sum a_{2i+1}^2 > 0,
\end{eqnarray*}
where we are using that for any $j$ we have $a_1 + a_3 + \cdots + a_{2j-1} \geq a_2 + a_4 + \cdots + a_{2j}$  since $i$ precedes $\overline{i}$ in $\overline{\B}$ for all $i$, and $a_1 + a_3 + \cdots + a_{2k-1} = a_2 + a_4 +  \cdots + a_{2k} = n$.
\end{proof}

\newpage

As shown in \cite{ArdilaEscobar}, the harmonic polytope $H_{n,n}$ is given by the following minimal inequality description:
\begin{eqnarray*}
\sum_{e \in [n]} x_e &=& \frac{n(n+1)}2 + 1, \\
\sum_{e \in [n]} y_e &=& \frac{n(n+1)}2 + 1, \\
\sum_{s \in S} x_s + \sum_{t \in T} y_t &\geq& \frac{|S|(|S|+1) + |T|(|T|+1)}2  + 1 \qquad \text{for each bisubset $S|T$ of $[n]$}.
\end{eqnarray*}
We translate it by the vector $-(\frac{n+1}2 + \frac1n)(\e_E + \f_E)$ so that it lands on the subspace $\M_n \times \M_n$ given by $\sum_{e \in [n]} x_e = \sum_{e \in [n]} y_e = 0$. Then we have the following statement.

\begin{proposition} \label{prop:harmonicineqs}
Let $f(x) = x\left(\frac{x-n}2-\frac1n\right)$.
The support function of the harmonic polytope $H_{n,n}$
\[
H(S|T) =  f(|S|) + f(|T|)  + 1
 \qquad \text{for each bisubset $S|T$ of $[n]$}
\]
satisfies the weak wall-crossing inequalities of the bipermutahedral fan.
\end{proposition}

\begin{proof}
Again, we already know this statement must be true because the normal fan of the harmonic polytope is a coarsening of the bipermutahedral fan \cite{ArdilaEscobar}, so its support function $H$ must be in the nef cone of $\Sigma_{n,n}$. However, giving a direct proof will allow us to derive a stronger result. 

\medskip

\noindent \textsf{A) Supermodular inequalities:}
Let $\B = b_1| \cdots |b_{h-1}| ij |  b_{h+1}| \cdots | b_{2n-2}$ 
be a bisequence of length $2n-2$ consisting of $2n-1$ singletons and one pair. Let $S=\{b_1, \ldots, b_{h-1}\}$ and $T=\{b_{h+1}, \ldots, b_{2n-1}\}$.  
Let $s=|S|, t=|T|,$ and $u = |S \cap T|$. As in the proof of Proposition \ref{prop:bipermineqs}, we consider three cases:

%

\smallskip
\noindent
(i) $i$ and $j$ appear on the same side of $ij$ in $\B$; say $i,j \in S$:
\smallskip

\noindent
\begin{eqnarray*}
I_{\B}(H) &:=&  \big(H(S|ijT) + H(Sij|T)\big) - \big(H(Si|Tj) + H(Sj|Ti)\big) \\
&=& 
\big(f(s) + f(t+2)\big) + \big(f(s) + f(t)\big) - \big(f(s)+f(t+1)\big) - \big(f(s)+f(t+1)\big) \\
&=& 1 > 0.
\end{eqnarray*}

\smallskip
\noindent
(ii) $i$ and $j$ appear on different sides of $ij$ in $\B$; say $i \in S$ and $j \in T$:
\smallskip

\noindent
\begin{eqnarray*}
I_{\B}(H) &=&  
\big(f(s) + f(t+1)\big) + \big(f(s+1) + f(t)\big) - \big(f(s)+f(t)\big) - \big(f(s+1)+f(t+1)\big) \\
&=& 0.
\end{eqnarray*}

\smallskip
\noindent
(iii) Only one of $i$ and $j$ appears again in $\B$; say $i \in S$ and $j \notin S,T$:
\smallskip

\noindent
\begin{eqnarray*}
I_{\B}(H) &=&  
\big(f(s) + f(t+2)\big) + \big(f(s+1) + f(t)\big) - \big(f(s)+f(t+1)\big) - \big(f(s+1)+f(t+1)\big) \\
&=& 1 > 0.
\end{eqnarray*}

\medskip
\noindent
\textsf{B) Up-down inequalities}: Let $\B = b_1 | \ldots | {\bf i} | \ldots | {\bf j} | \ldots |b_{2n-2}$ be a bisequence of length $2n-2$ consisting of $2n-2$ singletons, with non-repeated elements $i$ and $j$. Let $\overline{\B} = b_1 | \ldots | i | \overline{i} | \ldots | j | \overline{j} | \ldots |b_{2n-2}$.

If the spine of the bipartite graph $\Gamma(\B)$ has vertex labels $E=T_1, S_1, T_2, S_2, \ldots, T_k, S_k=E$, then the up-down inequality reads
\begin{eqnarray*}
I_{\B}(H) &=&  H(S_1|T_1) - H(S_1|T_2) + H(S_2|T_2) - H(S_2|T_3) + \cdots - H(S_{k-1}|T_k) + H(S_k|T_k) \\
&=& 
- \big(f(|S_1| + f(|T_1|)+1\big) 
+ \big(f(|S_1| + f(|T_2|)+1\big)  
- \big(f(|S_2| + f(|T_2|)+1\big)  + \\
& & 
+ \big(f(|S_2| + f(|T_3|)+1\big) - \cdots 
+  \big(f(|S_{k-1}| + f(|T_k|)+1\big) 
- \big(f(|S_k| + f(|T_k|)+1\big) \\
&=& -f(n)-f(n)-1 \\
&=& 1>0.
\end{eqnarray*}
We conclude that $H$ satisfies all the wall-crossing inequalities.
\end{proof}

It is said that $Q$ is a \emph{weak Minkowski summand} of $P$ if the normal fan of $Q$ refines the normal fan of $P$; this is equivalent to the existence of a scalar $\lambda$ such that $\lambda Q$ is a Minkowski summand of $P$; that is, there exists a polytope $R$ such that $P = \lambda Q + R$. The following parameter makes the situation more precise.

\begin{definition}
If $P$ and $Q$ are polytopes in $\R^d$, we define their \emph{Minkowski quotient}
\[
P/Q = \max \{\lambda \geq 0 \, : \, \lambda Q \text{ is a Minkowski summand of } P\}.
\]
\end{definition}

Note that $Q$ is a weak Minkowski summand of $P$ if and only if $P/Q > 0$.

\begin{theorem} \label{th:Pi/H}
The Minkowski quotient of the bipermutahedron and the harmonic polytope is
\[
\Pi_{n,n} / H_{n,n} = 2
\]
for all integers $n \geq 2$.
\end{theorem}

\begin{proof}
If $\lambda H_{n,n}$ is a Minkowski summand of $\Pi_{n,n}$, we have $\Pi_{n,n} = \lambda H_{n,n} + R$ for a polytope $R$. Since $R$ is a Minkowski summand of $\Pi_{n,n}$, its normal fan coarsens the bipermutahedral fan, so $R$ is in the nef cone of the bipermutahedral fan $\Sigma_{n,n}$. It follows that its support function
\[
R: = \Pi - \lambda H
\]
satisfies the wall-crossing inequalities. Conversely, if $R = \Pi - \lambda H$ satisfies the wall-crossing inequalities, then it is the support function of a deformation $R$ of the bipermutahedron such that $\Pi_{n,n} = R + \lambda H_{n,n}$, so $\lambda H_{n,n}$ is a Minkowski summand of $\Pi_{n,n}$. We conclude that
\[
P/Q = \max \{\lambda \geq 0 \, : \, R = \Pi - \lambda H \text{ satisfies the wall-crossing inequalities of } \Sigma_{n,n}\}
\]
Looking back at the proofs of Propositions \ref{prop:bipermineqs} and \ref{prop:harmonicineqs} we obtain the following.

\smallskip

\noindent \textsf{A) Supermodular inequalities:}
In the three cases (i), (ii), (iii), we have
\[
(i) \,\, I_{\B}(R) = 2 - \lambda, \qquad
(ii) \,\, I_{\B}(R) = 2 , \qquad
(iii) \,\, I_{\B}(R) = 4 - \lambda, \qquad
\]

\noindent \textsf{B) Up-down inequalities:} We have
\[
(i) \,\, I_{\B}(R) \geq n - \lambda.
\]

\noindent
The largest $\lambda$ for which these numbers are non-negative is $2$, as desired.
\end{proof}

\section{Acknowledgments}

I would like to extend my warm gratitude to Graham Denham and June Huh for the very rewarding collaboration \cite{ArdilaDenhamHuh1} that led to the construction of the harmonic polytope and the bipermutahedron, and to Laura Escobar \cite{ArdilaEscobar} for the very rewarding collaboration that taught me most of what I know about the harmonic polytope. Several of the results in this paper were inspired by conversations with them.
In particular, Graham asked the question that led to Proposition \ref{prop:symmetry}, and programs that Graham and Laura wrote helped me discover Theorems \ref{th:biEulerianseries} and \ref{th:Pi/H}. I am also grateful to Katharina Jochemko, who taught me Wagner's technique used to prove Theorem \ref{th:realrooted}. Early conversations with M\'onica Blanco and Francisco Santos helped solidify my understanding of the bipermutahedron. 
Finally, as is often the case, the Online Encyclopedia of Integer Sequences at \texttt{www.oeis.org} was a very valuable tool for this project.

\small

\bibliographystyle{habbrv}
\bibliography{ref}

\end{document}